\documentclass[11pt,a4paper]{article}
\usepackage{amsfonts}

\usepackage{graphicx,setspace}
\usepackage{amsfonts,amsmath, amssymb}
\usepackage{epstopdf}
\usepackage{color,amsmath,amssymb, amsfonts,amstext,amsthm}
\usepackage{bbm}
\usepackage{mathrsfs}
\usepackage{cite}
\usepackage{amsmath}
\usepackage{enumitem}

\newtheorem{remark}{\textbf{\ \ \quad Remark}}[section]

\newtheorem{prop}{\textbf{\ \ \quad Proposition}}[section]

\renewcommand{\b}{\beta}

\newcommand{\lam}{\lambda}
\newcommand{\s}{\sigma}

\renewcommand{\phi}{\varphi}
\newcommand{\e}{\varepsilon}
\renewcommand{\a}{\alpha}
\renewcommand{\b}{\beta}

\newcommand{\R}{{\mathbb R}}
\newcommand{\tr}{\triangle}

\newcommand{\eps}{\varepsilon}

\newcommand{\be}{\begin{equation}}
\newcommand{\ee}{\end{equation}}
\newcommand\bes{\begin{eqnarray}}
\newcommand\ees{\end{eqnarray}}
\newcommand{\bess}{\begin{eqnarray*}}
\newcommand{\eess}{\end{eqnarray*}}

\newcommand{\dx}{{\rm d}x}
\newcommand{\dy}{{\rm d}y}
\newcommand{\dt}{{\rm d}t}

\newcommand{\befig}{\begin{figure}}
\newcommand{\enfig}{\end{figure}}

\newcommand{\bear}{\begin{eqnarray}}
\newcommand{\enar}{\end{eqnarray}}

\newcommand{\bearn}{\begin{eqnarray*}}
\newcommand{\enarn}{\end{eqnarray*}}
{\setlength\arraycolsep{2pt}}

\title{Fokker-Planck equation driven by asymmetric L\'evy motion
}
\author{Xiao Wang$^1$, Wenpeng Shang$^1$, Xiaofan Li$^2$\footnote{Corresponding author. E-mail: lix@iit.edu}, Jinqiao Duan$^2$, Yanghong Huang $^3$  \\
\\
\ \\
   {\small \it $^1$ School of Mathematics and Statistics, Henan University}\\
  {\small \it Kaifeng 475001, China }\\
  {\small \tt email xwang@vip.henu.edu.cn}\\
   {\small \it $^2$ Department of Applied Mathematics, Illinois Institute of Technology}\\
  {\small \it Chicago, IL 60616, USA }\\
  {\small \tt email lix@iit.edu(X. Li), duan@iit.edu(J. Duan)}\\
  {\small \it $^3$ Department of Applied Mathematics,   University}\\
  {\small \it School of Mathematics, The University of Manchester, Manchester, M13 9PL, UK }\\
   {\small \tt  email yanghong.huang@manchester.ac.uk  }
}
\begin{document}
\maketitle

\begin{abstract}
Non-Gaussian L\'evy noises are present in many models for understanding
underlining principles of physics, finance, biology and more.
In this work, we consider the Fokker-Planck equation(FPE) due to
one-dimensional asymmetric L\'evy motion, which is a nonlocal
partial differential equation.
We present an accurate numerical quadrature for the singular integrals in
the nonlocal FPE and develop a fast summation method to reduce the order
of the complexity from $O(J^2)$ to $O(J\log J)$ in one time-step,
where $J$ is the number of unknowns.
We also provide conditions under which the numerical schemes satisfy
maximum principle.
Our numerical method is validated by comparing with exact solutions
for special cases.
We also discuss the properties of the probability density functions
and the effects of various factors on the solutions, including
the stability index, the skewness parameter,
the drift term, the Gaussian and non-Gaussian noises and the domain size.

\end{abstract}

\baselineskip=15pt

\section{Introduction}

The Fokker-Planck equation (FPE) was first used to describe the Brownian motion of particles,
which gives the time evolution of the probability density function
for the systems\cite{Gardiner2004,FPE_Risken}.
For stochastic differential equations (SDEs) driven by Brownian motion,
the corresponding FPE is a second-order parabolic partial differential equation.
For some special cases, the analytic solutions
can be founded but in general people seek numerical solutions with well-developed numerical methods for such differential equations.
However, noisy fluctuations are usually non-Gaussian in nature,
and the investigation of the FPE induced by non-Gaussian noises
is still in its infancy.
Non-Gaussian noises are widely found to describe the phenomenon
in physics, biology, economics among other fields\cite{Nature10,Cartea07,Xu2013}.
In this work, we consider FPEs derived from stable L\'evy motion
because of its properties of 'heavy tail' and central limit theorem.

A L\'evy motion is a stochastic process which has independent and stationary increments, stochastically continuous
sample paths. It is completely determined by the L\'evy-Khintchine formula, i.e. a L\'evy motion is characterized
by the generating triplet $(b,A,\nu)$, $b$ is a drift vector, $A$ is a diffusion matrix and $\nu$
is an L\'evy measure satisfying $\int_{\mathbb{R}^n \setminus \{0\}}(|y|^2\wedge 1)\nu(\dy)<\infty$.

In this paper, we mainly consider FPEs corresponding to scalar SDEs
with stable L\'evy motion whose distribution $S_\a(\tau,\b,\mu)$
is determined by four parameters: the index of stability $\a$ ($0<\a\leq 2$),
the scaling parameter $\tau$ ($\tau\geq 0$),
the skewness parameter ($-1\leq \beta \leq 1$) and the shift parameter $\mu$.
The L\'evy-Khintchine formula
for the stable L\'evy motion is \cite{Sato-99, asymmetric_Hein}
\bear \label{Charfun}
    \mathbb{E}(e^{i\lam L_t}) =
    \begin{cases}
        \mbox{exp}\{-\tau^{\a}|\lam|^{\a}t(1-i\beta \mbox{sgn}{\lam} \tan{\frac{\pi \a}2}) + i\mu \lam t \},\; \mbox{for} \; \a \neq 1,\\
        \mbox{exp}\{-\tau|\lam|t(1+i\beta \frac{2}{\pi} \mbox{sgn}{\lam} \log{|\lam|}) + i\mu \lam t \}, \quad \mbox{for} \;\a = 1 .
    \end{cases}
\enar
The $\a$-stable random variable is strictly stable for $\mu=0$ and $\a\neq1$, but when $\a=1$, it is strictly stable if and only if the process
is symmetric ($\b=0$). We take $\s=1$ and $\mu=0$ for consideration.

Recently, numerous research focus on the symmetric stable
L\'evy motion, corresponding to $\beta=0$ in \eqref{Charfun},
partially because the infinitesimal generator of the
process is related to  the fractional Laplacian operator. The operator
has many equivalent definitions including singular integrals, the Riesz
potential operator, Bochner's subordination and so on \cite{Kwasnicki17}.
Schertzer \emph{et al.}\cite{Duan01} derived a fractional FPE of
nonlinear stochastic differential equations driven by non-Gaussian
L\'evy stable noises and discussed the existence and uniqueness
of the solution.
A number of references \cite{Wei15, Wang17} showed
the existence and uniqueness of weak solution
for nonlocal Fokker-Planck equations.
Huang \emph{et al.} \cite{qiao16} investigated weak and strong maximum principles
for a class of general L\'evy type Markov generators(nonlocal Waldenfels operators).
The regularity results for the solutions are given
in \cite{ROSOTON14, GRUBB15,Cozzi17}.

The numerical methods developed recently for the case include
finite differences with quadrature, spectral method, Galerkin finite element
method.
Using the different definitions of fractional derivatives, Liu \emph{et al.}\cite{Liu2004}
transformed the space fractional Fokker-Planck equation into ordinary differential equations and
solved it by a method of lines.
Huang \emph{et al.} \cite{Huang14}, Gao \emph{et al.} \cite{Ting12}
and Wang \emph{et al.} \cite{METXiao2015} presented finite difference methods
with different quadrature rules in one and two dimensions.
Du \emph{et al.} \cite{Du2012, TD15} considered the discontinuous
and continuous Galerkin methods for certain nonlocal diffusion problems.
D'Elia \emph{et al.} \cite{DELIA13} used the nonlocal vector calculus
to show that the solutions of these nonlocal problems converge to the
solutions of fractional Laplacian problems
as the domain of the nonlocal interactions becomes infinite.
Mao \emph{et al.} \cite{FPDE_Shen, Mao17} developed efficient spectral-Galerkin
algorithms for fractional partial differential equations.
Recently, Acosta \emph{et al.} \cite{Acosta17}
dealt with the integral version of the Dirichlet homogeneous fractional
Laplace equation and gave a finite element analysis. They also presented
high-order numerical methods for one-dimensional fractional-Laplacian
boundary value problems \cite{AcostaMC2017}.

 Asymmetric $\a$-stable L\'evy motion is widely applied in physical sciences
and economy \cite{Mengli,asy_envirment,first}.
Chen \emph{et al.} \cite{Chen2016, CHEN17} constructed and proved
the existence and the uniqueness of the fundamental
solution (heat kernel) of nonsymmetric L\'evy-type operator
and established its sharp two-sided estimates.
Riabiz \emph{et al.} \cite{Riabiz17} gives the modified Poisson series
representation of linear stochastic processes
driven by asymmetric stable L\'evy process.
As the lack of the closed-form density expressions limits its application,
we will provide the numerical approximation for the probability density as solution to Fokker-Planck equation.
However, to our knowledge, there are few work in developing the numerical methods for the asymmetric case.
Zeng \emph{et al.} \cite{Zeng10} presented the numerical solution of the space-fractional FPE and studied properties
of parameter-induced aperiodic stochastic resonance in the presence of asymmetric L\'evy noise.

In our work, we mainly consider the FPE driven by asymmetric L\'evy motion,
develop a fast numerical scheme and discuss the properties of
the probability density functions.
This paper is structured as follows.
In section 2, we present the Fokker-Planck equation for a SDE driven by
asymmetric $\a$-stable L\'evy motion.
In section 3, we show the symmetry of solutions and present
the numerical scheme. The numerical solutions and their properties
are shown in section 4.
Finally, we summarize the results in section 5.
%Say: sample path simulation or Monte Carlo simulation

\section{Fokker-Planck equation driven by stable L\'evy motion    }
\setcounter{equation}{0}

Consider the following SDE
\begin{equation} \label{SDE01}
\mathrm{d}X_t=f(X_t)\mathrm{d}t+\sigma\mathrm{d}B_t+\mathrm{d}L_t, \quad X_0=x.
\end{equation}
where $f$ is a drift term, $\sigma$ is non-negative diffusion constant,
$L_t$ is a L\'evy motion with the generating triplet $(\e K_{\a,\beta},0,\e\nu_{\a,\b})$ and $\nu_{\a,\b}$ is a L\'evy jump measure
to be defined later. We choose the L\'evy motion $L_t$ such that $L_1$ is the random variable whose the probability density
function(PDF) corresponds to the stable distribution $S_\a(1,\b,0)$.
Here, $\a$ ($0<\a\leq2$) is the index of stability and $\e$ is the intensity of L\'evy noise.

For every $\phi \in H_0^2(\mathbb{R})$, the generator for the solution of SDE (\ref{SDE01}) is
\bear \label{oper}
   \mathscr{A}\phi=( f(x)+\e K_{\a,\b})\phi_x+\frac{\sigma^2}{2}\phi_{xx}+\varepsilon \mathscr{L}_{\a,\b}\phi,
\enar
where
\begin{eqnarray} \label{operator1}
    \mathscr{L}_{\a,\beta}\phi&=&\int_{\mathbb{R}\setminus\{0\}}
    [\phi(x+y)-\phi(x)-1_{\{|y| <1\}}y\phi'(x)]\nu_{\a,\b}(\mathrm{d}y).
\end{eqnarray}

The Fokker-Planck equation for the PDF $p(x,t)$ of the process $X_t$ associated with SDE (\ref{SDE01}) is
\begin{equation}
   p_t= \mathscr{A}^*p, \qquad p(x,0)=p_0(x),
\label{eq.fpeas}
\end{equation}
where $ \mathscr{A}^*$ is the adjoint operator of $\mathscr{A}$ in Hilbert space\cite{Duanbook2}, given by
\iffalse
\begin{eqnarray*}
   \int_\mathbb{R}(\mathscr{L} p)(x)\varphi(x)\mathrm{d}x&=&\int_\mathbb{R}\int_{\mathbb{R}\setminus\{0\}}
   [p(x+y)-p(x)-1_{\{|y|<1\}}yp'(x)]\varphi(x)\nu(\mathrm{d}y)\mathrm{d}x \\
   &=&\int_{\mathbb{R}\setminus\{0\}}\int_\mathbb{R}
   [p(x+y)-p(x)-1_{\{|y|<1\}}yp'(x)]\varphi(x)\mathrm{d}x\nu(\mathrm{d}y)\\
   &=&\int_{\mathbb{R}\setminus\{0\}}\int_\mathbb{R}p(x)
   [\varphi(x-y)-\varphi(x)+1_{\{|y| <1\}}y\varphi'(x)]\mathrm{d}x\nu_{\a,\b}(\mathrm{d}y) \\
   &=&-\int_{\mathbb{R}}p(x)\int_{\mathbb{R}\setminus\{0\}}
   [\varphi(x)-\varphi(x-y)-1_{\{|y| <1\}}y\varphi'(x)]\nu(\mathrm{d}y)\mathrm{d}x \\
   &=&\int_{\mathbb{R}}p(x)(\mathscr{L}^* \varphi)(x)\mathrm{d}x \\
\end{eqnarray*}
\fi
\be \label{jointOper}
   \mathscr{A}^*\varphi = -((f+ \e K_{\a,\b})\varphi)_x+\frac{\sigma^2}{2}\phi_{xx}-\e \int_{\mathbb{R}\setminus\{0\}}
   [\varphi(x)-\varphi(x-y)-1_{\{|y| <1\}}y\varphi'(x)]\nu_{\a,\b}(\mathrm{d}y).
\ee

%-----------------------------------------------------------------------------------------------
\iffalse
\subsection{Symmetric $\alpha$-stable L\'evy motion}

If the $L_t$ is symmetric $\a$-stable L\'evy motion with L\'evy jump measure \\
\begin{equation*}
\nu_\a(\mathrm{d}y)=\frac{C_\a}{|y|^{1+\a}}\mathrm{d}y, \;C_\a=\frac{\a}{2^{1-\a}\sqrt{\pi}}\frac{\Gamma(\frac{1+\a}{2})}{\Gamma(1-\frac{\a}{2})},
\end{equation*}
the $\mathscr{L}$ is fractional Laplacian operator which is self-adjoint.
Then, the Fokker-Planck equation for the process $X_t$ is the following,
\begin{equation} \label{FPKsym}
   p_t=-(f(x)p)_x+\frac{\sigma^2}{2}p_{xx}+\varepsilon\int_{\mathbb{R}\setminus \{0\}}[p(x+y)-p(x)-1_{\{|y|<1\}}yp_x]\nu_\a(\mathrm{d}y),
\end{equation}
the last integral term is defined into the fractional Laplacian operator $C_{\a}(-\Delta)^\frac{\a}{2}$, which is also equal to the
fractional derivatives,
\bess
 (-\Delta)^{\frac{\a}{2}}u(x) = \frac{{}_{-\infty}D_x^\a u(x) + {}_xD_{\infty}^\a u(x)}{2\cos(\frac{\pi\a}{2})},
\eess
where ${}_{-\infty}D_x^\a$ and ${}_xD_{\infty}^\a $ are left and right Riemann-Liouville fractional derivatives\cite{GM98, Huang}, so
there are different numerical schemes for the Eq.~(\ref{FPKsym}).
\fi
%-----------------------------------------------------------------------------------------------

\subsection{Asymmetric $\alpha$-stable L\'evy motion}
In the following, we focus on the FPE driven by asymmetric $\a$-stable L\'evy motion.

The L\'evy measure $\nu_{\a,\beta}$ is given by \cite{Xiao2016}
\bear  \label{asy_measure}
   \nu_{\a,\beta}({\rm d}y) = \frac{C_p 1_{\{0<y<\infty\}}(y)+C_n 1_{\{-\infty<y<0\}}(y)}{|y|^{1+\alpha}} {\rm d}y,
\enar
with
\bear \label{C1C2}
     C_p = C_{\a}\frac{1+\beta}2, \quad C_n = C_{\a}\frac{1-\beta}2,\quad -1 \leq \beta \leq 1,
\enar
and
\bear  \label{C_alpha}
    C_{\a} =
    \begin{cases}
        \frac{\a(1-\a)}{\Gamma(2-\a)\cos{(\frac{\pi \a}2)}}\;,   &\text{ $ \a \neq 1 $;}\\
        \frac2{\pi},  \;  &\text{ $ \a = 1$.}
    \end{cases}
\enar
The constant $K_{\a,\beta}$ in \eqref{oper} is given by
\bear  \label{K}
    K_{\a,\beta}=
    \begin{cases}
        \frac{C_p-C_n}{1-\a}=\frac{\beta C_\a}{1-\a} \;,   &\text{ $ \a \neq 1 $;}\\
        (\int_1^{\infty}\frac{\sin(x)}{x^2}{\rm d}x + \int_0^{1}\frac{\sin(x)-x}{x^2}{\rm d}x)(C_n-C_p),  \;  &\text{ $ \a = 1$.}
    \end{cases}
\enar
Here notice that $C_p(-\b) = C_n(\beta)$ and $K_{\a,-\b}= -K_{\a,\b}$.

We can rewrite the adjoint operator of $\mathscr{L}_{\a,\b}$
\begin{eqnarray*}
  -\mathscr{L}_{\a,\b}^*\varphi(x)=\int_{\mathbb{R}\setminus\{0\}}
  [\varphi(x)-\varphi(x-y)-1_{\{|y|<1\}}y\varphi'(x)]\nu_{\a,\b}(\mathrm{d}y)
\end{eqnarray*}
as
\begin{eqnarray} \label{adjoint}
  -\mathscr{L}_{\a,\b}^*\varphi(x)= -\int_{\mathbb{R}\setminus\{0\}}[\varphi(x+y)-\varphi(x)-1_{\{|y|<1\}}y\varphi'(x)]\nu_{\a,-\b}(\mathrm{d}y).
\end{eqnarray}
by making the change of integration variable $y\rightarrow -y$.

Finally, the adjoint operator of the nonlocal term in Eq.~(\ref{jointOper}) is
\be \label{asy_generator}
    \mathscr{L}_{\a, \b}^* \varphi(x) = \int_{\mathbb{R}\setminus\{0\}}[\varphi(x+y)-\varphi(x)-1_{\{|y|<1\}}y\varphi'(x)]\nu_{\a,-\b}(\mathrm{d}y).
\ee

For numerical computations, we change $1_{\{|y|<1\}}$ to $1_{\{|y|<b\}}$ in (\ref{asy_generator}).
Then, the Fokker-Planck equation driven by the asymmetric L\'evy motion becomes

\begin{equation}\label{FPE1Dasy}
   p_t=-(c(x)p)_x+\frac{\sigma^2}{2}p_{xx}+\varepsilon\int_{\mathbb{R}\setminus \{0\}}
   [p(x+y,t)-p(x,t)-1_{\{|y|<b\}}yp_x]\nu_{\a, -\b}(\mathrm{d}y),
\end{equation}
where
\bear
    c(x) =
    \begin{cases}
        f(x) + \eps K_{\a,\b} + \eps(C_p-C_n)\frac{b^{1-\a}-1}{1-\a}\;,   &\text{ $ \a \neq 1 $;}\\
        f(x) + \eps K_{\a,\b} + \eps(C_p-C_n)\ln{b},  \;  &\text{ $ \a = 1$,}
    \end{cases}
\label{eq.cs}
\enar
with the constants $C_p$ and $C_n$ defined by (\ref{C1C2}).

\begin{remark}
If $\beta=0$, the process $L_t$ becomes symmetric L\'evy motion,
then $c(x)=f(x)$,
$\displaystyle{\nu_{\a, -\b}(\mathrm{d}y)=\frac{C_\a}{|y|^{1+\a}} \dy}$ and $\mathscr{L}_{\a,\b}$ is self-adjoint.
\end{remark}

\subsection{Auxiliary conditions}\label{sec.ac}
 For the solution of Eq.~(\ref{FPE1Dasy}), we specify auxiliary conditions. There are several various boundary conditions
 for Brownian motion, such as reflecting barrier,  absorbing barrier, periodic boundary condition and so on\cite{Gardiner2004}.
 We consider two cases of processes governed by the SDE (\ref{SDE01}). One is that the process $X_t$ disappears or is killed when $X_t$ is outside
 a bounded domain $D=(-b,b)$. In this case, we have the absorbing condition, i.e. the probability $p(x,t)$
of being outside of the bounded domain $D=(-b,b)$ is zero:
\bear
    p(x,t) = 0, \;\;\; x \notin (-b,b).
\label{eq.ec}
\enar
 We can also extend the domain to general unsymmetric domains.

The other case is that the process $X_t$ can go anywhere on the entire real line $\mathbb{R}$.
In this case, we call it the natural condition. The probability density satisfies
\bess
    \int_{-\infty}^{\infty} p(x,t) \dx  = 1, \;\;\; \forall t\geq 0.
\eess

\section{Numerical schemes}
\setcounter{equation}{0}
\subsection{Simplification}
First, we show the following symmetry for the solution of FPE~(\ref{FPE1Dasy}).

\begin{prop} [Symmetry of Solutions]

If $f(x)$ is an odd function and the domain $D$ is symmetric
($D=(-b,b)$), then the solution $p(x,t)$ of the Fokker-Planck eqaution~(\ref{FPE1Dasy})
is symmetric about the origin for any given time $t$ if $\beta$ changes the sign, i.e. $p(-x,t;-\beta) = p(x,t;\beta)$
for all $x\in (-b, b)$ where $p(x,t;\beta)$ and $p(x,t;-\beta)$ denote the solutions corresponding to $\beta$ and $-\beta$ respectively.
\end{prop}

\emph{Proof.}
   From the Eq.~(\ref{FPE1Dasy}), $p(-x,t;-\beta)$ satisfies
\bear
   p_t(-x,t;-\beta)&=&-(cp)_x(-x,t;-\b)+\frac{\sigma^2}{2}p_{xx}(-x,t;-\beta) \\
 && +\varepsilon\int_{\mathbb{R}\setminus \{0\}}
   [p(-x+y,t;-\beta)-p(-x,t;-\beta)-1_{\{|y|<b\}}yp_x(-x,t;-\beta)]\nu_{\a, \b}(\mathrm{d}y).  \nonumber
\enar
By the definition of $c(x)$ in (\ref{FPE1Dasy}), we have $c(-x;-\beta) = - c(x;\beta)$ if $f(x)$ is an odd function.\\
Denote $\tilde{p}(x,t) = p(-x,t;-\beta)$, then
\bear
    \tilde{p}_x(x,t)=-p_x(-x,t;-\beta),\; \tilde{p}_{xx}(x,t)=p_{xx}(-x,t;-\beta).
\enar \label{th11}
Taking $y' = -y$, we have
\bear \label{eq:th1}
    &&\int_{\R \setminus\{0\}} [p(-x+y,t;-\b)-p(-x,t;-\b) -  I_{\{|y|< b \}}(y) \; y p_x(-x,t;-\b)] \nu_{\a, \beta}( {\rm d}y) \nonumber\\
     &=& \int_{\R \setminus\{0\}} [\tilde{p}(x-y,t)-\tilde{p}(x,t) +  I_{\{|y|< b \}}(y) \; y \tilde{p}_x(x,t)]
        \left[\frac{C_p 1_{\{0<y<\infty\}}+C_n 1_{\{-\infty<y<0\}}}{|y|^{1+\alpha}}\right]\; {\rm d}y  \nonumber\\
     &=& \int_{\R \setminus\{0\}} [\tilde{p}(x+y',t)-\tilde{p}(x,t) -  I_{\{|y'|< b \}}(y') \; y' \tilde{p}_x(x,t)]
       \nu_{\a,-\b}({\rm d}y').
\enar
From the property of $c(x)$, Eqs.~(\ref{th11}) and (\ref{eq:th1}), we finally have
\bear
   \tilde{p}_t(x,t)&=&-(c(x;\beta)\tilde{p}(x,t))_x+\frac{\sigma^2}{2}\tilde{p}_{xx}(x,t) \\
 && +\varepsilon\int_{\mathbb{R}\setminus \{0\}}
   [\tilde{p}(x+y,t)-\tilde{p}(x,t)-1_{\{|y|<b\}}y\tilde{p}_x(x,t)]\nu_{\a, -\b}(\mathrm{d}y)  \nonumber
\enar
Then, we get the conclusion by the uniqueness of the solution to (\ref{FPE1Dasy}).
\begin{flushright}
   $ \Box $
\end{flushright}

For convenience, we convert the domain $D=(-b,b)$ to the standard domain $(-1,1)$ by the variable transformation

\begin{equation}
s = x/b, \quad \text{ and } v(s,t) := p(bs,t).
\label{eq.cov}
\end{equation}

Then, Eq.~(\ref{FPE1Dasy}) changes to
\begin{eqnarray}  \label{asymmetricEq3}
  v_t=&-&\frac{1}{b} (c(bs)v)_s + \frac{\s^2}{2b^2} v_{ss} \\
 &+&\eps b^{-\a} \int_{\R \setminus\{0\}} [v(s+r,t)-v(s,t)
     -  I_{\{|r|< 1 \}}(r) \; r v_s(s,t)]
    \left[\frac{C_n 1_{\{0<r<\infty\}}+C_p 1_{\{-\infty<r<0\}}}{|r|^{1+\alpha}}\right] {\rm d}r  \nonumber.
\end{eqnarray}

Next, we present the numerical schemes for the absorbing boundary condition.
By using the absorbing boundary condition, i.e. the probability density
$v(x,t)$ vanish outside of the standard domain $D=(-1,1)$. The above equation (\ref{asymmetricEq3}) containing the  singular integral
is  simplified to the following equations\cite{Xiao2016}.

For $s<0$,
\begin{eqnarray} \label{fpk_asy1}
    v_t= &&\frac{\sigma^2}{2b^2}v_{ss}-\frac{1}{b}(c(bs)v)_s+\varepsilon b^{-\a} C_pg(s)v_s-\varepsilon b^{-\a}\frac{v}{\a}\left[\frac{C_n}{(1-s)^\a}
    +\frac{C_p}{(1+s)^\a}\right]  \nonumber\\
    &+&\varepsilon b^{-\a}C_n \int_0^1\frac{v(s+y,t)-v(s,t)-yv_s(s,t)}{y^{1+\a}}\mathrm{d}y+
    \varepsilon b^{-\a}C_n\int_1^{1-s}\frac{v(s+y,t)-v(s,t)}{y^{1+\a}}\mathrm{d}y  \nonumber\\
    &+&\varepsilon b^{-\a}C_p \int_0^{1+s}\frac{v(s-y,t)-v(s,t)+yv_s(s,t)}{y^{1+\a}}\mathrm{d}y;
\end{eqnarray}

For $s\geq0$,
\begin{eqnarray} \label{fpk_asy2}
    v_t= &&\frac{\sigma^2}{2b^2}v_{ss}-\frac{1}{b}(c(bs)v)_s-\varepsilon b^{-\a} C_ng(s)v_s-\varepsilon b^{-\a} \frac{v}{\a}\left[\frac{C_n}{(1-s)^\a}
    +\frac{C_p}{(1+s)^\a}\right] \nonumber\\
    &+&\varepsilon b^{-\a} C_n \int_0^{1-s}\frac{v(s+y,t)-v(s,t)-yv_s(s,t)}{y^{1+\a}}\mathrm{d}y+
    \varepsilon b^{-\a} C_p\int_1^{1+s}\frac{v(s-y,t)-v(s,t)}{y^{1+\a}}\mathrm{d}y \nonumber\\
    &+&\varepsilon b^{-\a} C_p \int_0^1\frac{v(s-y,t)-v(s,t)+yv_s(s,t)}{y^{1+\a}}\mathrm{d}y.
\end{eqnarray}
Here,
\bear
    g(s) =
    \begin{cases}
        \frac{1-(1-|s|)^{1-\a}}{1-\a}\;,   &\text{ $ \a \neq 1 $;}\\
         -\ln(1-|s|),  \;  &\text{ $ \a = 1$.}
    \end{cases}
\label{g1}
\enar

Assuming the drift term $f$ is differentiable, we change the Eqs.~(\ref{fpk_asy1}) and (\ref{fpk_asy2})
to the following.

For $s<0$,
\begin{eqnarray} \label{FPKchange1}
    v_t= &&\frac{\sigma^2}{2b^2}v_{ss}(s,t)-m_1(s)v_s(s,t)- m_2(s)v(s,t) \nonumber\\
    &+&\varepsilon b^{-\a}C_n \int_0^1\frac{v(s+y,t)-v(s,t)-yv_s(s,t)}{y^{1+\a}}\mathrm{d}y+
    \varepsilon b^{-\a}C_n\int_1^{1-s}\frac{v(s+y,t)-v(s,t)}{y^{1+\a}}\mathrm{d}y  \nonumber\\
    &+&\varepsilon b^{-\a}C_p \int_0^{1+s}\frac{v(s-y,t)-v(s,t)+yv_s(s,t)}{y^{1+\a}}\mathrm{d}y;
\end{eqnarray}

for $s\geq0$,
\begin{eqnarray} \label{FPKchange2}
    v_t=  &&\frac{\sigma^2}{2b^2}v_{ss}(s,t)-m_1(s)v_s(s,t)- m_2(s)v(s,t) \nonumber \\
    &+&\varepsilon b^{-\a} C_n \int_0^{1-s}\frac{v(s+y,t)-v(s,t)-yv_s(s,t)}{y^{1+\a}}\mathrm{d}y+
    \varepsilon b^{-\a} C_p\int_1^{1+s}\frac{v(s-y,t)-v(s,t)}{y^{1+\a}}\mathrm{d}y \nonumber\\
    &+&\varepsilon b^{-\a} C_p \int_0^1\frac{v(s-y,t)-v(s,t)+yv_s(s,t)}{y^{1+\a}}\mathrm{d}y.
\end{eqnarray}
where
\bear
    m_1(s) =
    \begin{cases}
        \frac{c(bs)}{b}-\varepsilon b^{-\a}C_p g(s),   &\text{ $ s<0 $;}\\
         \frac{c(bs)}{b}+\varepsilon b^{-\a}C_n g(s),  \;  &\text{ $ s\geq0$,}
    \end{cases}
\label{m1}
\enar
and
\bear
    m_2(s) = c'(bs)+ \frac{\varepsilon b^{-\a}}{\a}\left[\frac{C_n}{(1-s)^\a}
    +\frac{C_p}{(1+s)^\a}\right].
\label{eq.m2}
\enar
\subsection{Discretization}

We aim to solve the FPE \eqref{FPKchange1} and \eqref{FPKchange2} numerically.
First, we denote $V_j$ as the numerical solution of $v$ at $(s_j,t)$,
where $s_j=jh$ for $j\in \mathbb{Z}$ and $h = \frac1J$.
Due to the absorbing condition, $V_j=0$ for $|j|\leq J$
and we denote the unknowns $V_j$ for $-J< j< J$
by the vector $\mathbf{V}:= (V_{-J+1}, V_{-J+2}, \cdots, V_{J-1})^T$.
Second, we approximate the diffusion term by the second-order central differencing and the first-order derivatives by the first-order upwind scheme.
Denoting
\bear
     \delta_u V_j=
    \begin{cases}
       \frac{V_{j}-V_{j-1}}{h}\; ,   & \text{ if $ m_1(s_j) > 0 $;}\\
       \frac{V_{j+1}-V_{j}}{h}\; ,   & \text{ if $ m_1(s_j) < 0 $},
    \end{cases}
\label{eq.uw}
\enar
we discretize the non-integral terms in the RHS of Eqs.~\eqref{FPKchange1} and \eqref{FPKchange2} as
\bear
(B\mathbf{V})_j:= C_h\frac{V_{j+1}-2V_j+V_{j-1}}{h^2} - m_1(s_j) \delta_u V_j
     -m_2(s_j)V_j,
\label{eq.B}
\enar
where
\begin{equation}
   C_h=\frac{\sigma^2}{2b^2}-\frac{\varepsilon b^{-\a}}{2} C_\a \zeta(\a-1)h^{2-\a}.
\label{eq.Ch}
\end{equation}
The second term in $C_h$ is the leading-order
correction term for the trapezoidal rules of the singular integrals in
Eqs.~(\ref{FPKchange1}) and (\ref{FPKchange2}) given below,
and $\zeta$ is the Riemann zeta function.
Eq.~\eqref{eq.B} defines a linear operator or a $(2J-1)$-by-$(2J-1)$ matrix $B$.

Third, the integrals in Eqs.~(\ref{FPKchange1}) and (\ref{FPKchange2})
are approximated by the trapezoidal rule
\bear
(S\mathbf{V})_j :=
    \begin{cases}
     & \eps b^{-\a} C_n h\sum\!{'}_{k=j+1}^{J+j} \frac{V_k-V_j-(s_k-s_j)\frac{V_{j}-V_{j-1}}{h}}{(s_k-s_j)^{\a+1}}
    + \eps b^{-\a} C_n h\sum_{k=J+j}^{J}\;\;\!\!\!\! \frac{V_k-V_j}{(s_k-s_j)^{1+\a}} \\
    &+ \eps b^{-\a} C_p h\sum\!{''}_{k=-J}^{j-1} \frac{V_k-V_j-(s_k-s_j)\frac{V_{j+1}-V_{j}}{h}}{(s_j-s_k)^{\a+1}},
     \quad \text{ for $-J+1\leq j\leq -1$,}  \\
     &  \eps b^{-\a} C_n h\sum\!{'}_{k=j+1}^{J} \frac{V_k-V_j-(s_k-s_j)\frac{V_{j}-V_{j-1}}{h}}{(s_k-s_j)^{\a+1}}
	    + \eps b^{-\a} C_p h\sum_{k=-J}^{-J+j}\;\;\!\!\!\! \frac{V_k-V_j}{(s_j-s_k)^{1+\a}}  \\
    &+ \eps b^{-\a} C_p h\sum\!{''}_{k=-J+j}^{j-1} \frac{V_k-V_j-(s_k-s_j)
    \frac{V_{j+1}-V_{j}}{h}}{(s_j-s_k)^{\a+1}},
     \quad \text{for $0\leq j\leq J-1$ .} \\
    \end{cases}
\label{eq.S}
\enar
The summation symbol $\sum$ means the terms of both end indices are multiplied by $\frac12$,
$\sum\!{'}$($\sum\!{''}$ ) means that only the term of the top (bottom) index is multiplied by $\frac12$.
Eq.~\eqref{eq.S} defines another linear operator or $(2J-1)$-by-$(2J-1)$ matrix $S$.
Note the trapezoidal rule in \eqref{eq.S} would induce significant error due
to the singular nature of the integrals and the dominant error is eliminated
by the second term in \eqref{eq.Ch}.  \cite{Xiao2016,Sidi}

Now, we present our semi-discrete scheme for solving the FPE \eqref{FPKchange1} and (\ref{FPKchange2})
\begin{equation}
	\frac{{\rm d}V_j}{{\rm d}t} =  (A\mathbf{V})_j,  \quad \text{ where }
	A := B + S,
\label{eq.FPEsd}
\end{equation}
for $-J+1 \leq j \leq J-1$.

\begin{remark}\label{rem.nc}
We point out that our semi-discrete scheme for the natural condition will
be the same as (\ref{eq.FPEsd}) except $m_2(s)$ simplifies to $c'(bs)$.
\end{remark}

\subsection{Fast Algorithm}

The summation terms in the discretized scheme \eqref{eq.FPEsd}
can be written in matrix-vector multiplication form
$S\mathbf{V}$ as given by \eqref{eq.S}.
We decompose the matrix $S$ as the summation of a Toeplitz matrix $T_S$ and a tridiagonal matrix $D_S$
\bear
    S = T_S + D_S,
\label{eq.dS}
\enar
where
\bear
T_S =
\begin{pmatrix}
0 & \frac{\tilde{C_n}}{h^{1+\a}} & \frac{\tilde{C_n}}{(2h)^{1+\a}} &\cdots & \frac{\tilde{C_n}}{[(2J-2)h]^{1+\a}}\\
\frac{\tilde{C_p}}{h^{1+\a}} & 0 &\frac{\tilde{C_n}}{h^{1+\a}} &\cdots & \frac{\tilde{C_n}}{[(2J-3)h]^{1+\a}}\\
\frac{\tilde{C_p}}{(2h)^{1+\a}} & \frac{\tilde{C_p}}{h^{1+\a}} & 0  & \ddots & \frac{\tilde{C_n}}{[(2J-4)h]^{1+\a}} \\
\vdots & \vdots & \ddots & \ddots & \vdots\\
 \frac{\tilde{C_p}}{[(2J-2)h]^{1+\a}} &  \frac{\tilde{C_p}}{[(2J-3)h]^{1+\a}} & \frac{\tilde{C_p}}{[(2J-4)h]^{1+\a}} & \cdots & 0
\end{pmatrix},
\enar
\iffalse
$$
D_{A_i} =
\begin{pmatrix}
 a_i & -b_i & &   \\
  b_i & \ddots & \ddots   &\\
  & \ddots  & \ddots & -b_i\\
 &   & b_i & a_i\\
\end{pmatrix}, \quad i =1,2.
$$
\fi
\bear
D_{S} =
\begin{pmatrix}
 a_{1} & p_{1} &    & &\\
  b_{2} & a_2 & p_{2}  & &\\
  & \ddots  & \ddots & \ddots &\\
  & &b_{2J-2} & a_{2J-2}& p_{2J-2} \\
 & &  & b_{2J-1} & a_{2J-1}\\
\end{pmatrix}.
\enar
\iffalse
\bess
D_{A_2} =
\begin{pmatrix}
b_{2(J+1)} & a_{J+1} & -b_{2(J+1)} & &   \\
 & b_{2(J+2)} & \ddots & \ddots   &\\
 & & \ddots  & \ddots & -b_{2(2J-2)}\\
 & &   & b_{2(2J-1)} & a_{2J-1}\\
\end{pmatrix}
\eess
\fi
Here
\bear
   \tilde{C_p} = \eps b^{-\a}C_p h, \quad \tilde{C_n} = \eps b^{-\a}C_n h,
\enar
and
\iffalse
\bess
    a_j &=& - \tilde{C_n} \Big[\sum_{k=j+1}^{J}\!\!\!{'} \frac{1}{(s_k-s_j)^{\a+1}}\Big] -
     \tilde{C_p}  \Big[\sum_{k=-J}^{j-1}\!\!\!{''} \frac{1}{(s_j-s_k)^{\a+1}}\Big]  , \quad\text{for}\; j=1,2,\cdots,2J-1,\\
    b_{j} &=&
    \begin{cases}
       \frac{\tilde{C_n} }{2h} \sum_{k=j+1}^{J+j}\!\!\!{'} \frac{1}{(s_k-s_j)^{\a}}
   -\frac{\tilde{C_p}}{2h} \sum_{k=-J}^{j-1}\!\!\!{''} \frac{1}{(s_j-s_k)^{\a}},\qquad\text{for}\; j=1,2,\cdots,J,\\
       \frac{\tilde{C_n} }{2h} \sum_{k=j+1}^{J}\!\!\!{'} \frac{1}{(s_k-s_j)^{\a}}
   -\frac{\tilde{C_p} }{2h} \sum_{k=-J+j}^{j-1}\!\!\!{''} \frac{1}{(s_j-s_k)^{\a}},\quad\text{for} \;j=J+1,J+2,\cdots,2J-1.
    \end{cases}
\eess \label{bj}
\fi
\bess
    a_{J+j} &=& - \tilde{C_n} \Big[\sum_{k=j+1}^{J}\!\!\!{'} \frac{1}{(s_k-s_j)^{\a+1}}\Big] -
     \tilde{C_p}  \Big[\sum_{k=-J}^{j-1}\!\!\!{''} \frac{1}{(s_j-s_k)^{\a+1}}\Big]  \\
     &&- \begin{cases}
       \frac{\tilde{C_n} }{h} \sum{'}_{k=j+1}^{J+j}  \frac{1}{(s_k-s_j)^{\a}}
   +\frac{\tilde{C_p}}{h} \sum{''}_{k=-J}^{j-1} \frac{1}{(s_j-s_k)^{\a}},\qquad\text{for}\; j=1-J,2-J,\cdots,0,\\
       \frac{\tilde{C_n} }{h} \sum{'}_{k=j+1}^{J} \frac{1}{(s_k-s_j)^{\a}}
   +\frac{\tilde{C_p} }{h} \sum{''}_{k=-J+j}^{j-1} \frac{1}{(s_j-s_k)^{\a}},\quad\text{for} \;j=1,J+2,\cdots,J-1,
    \end{cases}
\eess
\bess
    p_{J+j} &=&
    \begin{cases}
      \frac{\tilde{C_p}}{h} \sum{''} _{k=-J}^{j-1}\frac{1}{(s_j-s_k)^{\a}},\qquad\text{for}\; j=1-J,2-J,\cdots,0,\\
       \frac{\tilde{C_p} }{h} \sum{''}_{k=-J+j}^{j-1} \frac{1}{(s_j-s_k)^{\a}},\quad\text{for} \;j=1,2,\cdots,J-1,
    \end{cases}
\eess \label{bj1}
\bess
    b_{J+j} &=&
    \begin{cases}
       \frac{\tilde{C_n} }{h} \sum{'}_{k=j+1}^{J+j} \frac{1}{(s_k-s_j)^{\a}},\qquad\text{for}\; j=1-J,2-J,\cdots,0,\\
        \frac{\tilde{C_n} }{h} \sum{'}_{k=j+1}^{J} \frac{1}{(s_k-s_j)^{\a}},\quad\text{for} \;j=1,2,\cdots,J-1.
    \end{cases}
\eess \label{bj2}

The direct summation of $S\mathbf{V}$ causes the computational complexity
of the scheme \eqref{eq.FPEsd} to be $O(J^2)$. Realizing that
the dominant computational cost comes from Toeplitz matrix-vector
multiplications, we develop our fast algorithm based on
the well-known algorithm of $O(J\log J)$ for multiplying a vector
by a Toeplitz matrix \cite{golub2012matrix}.

Next, we compare the CPU times using fast algorithm with those
using direct summation in Table~\ref{RunTime}
for the case of $\a=0.5, \beta=0.5, f\equiv 0, d=0,\e=1$ with $\triangle t=\frac{1}{1600}$ and different
resolutions $J$ from time $0$ to $ t_F=1$. The initial condition
is a Gaussian density function $p(x,0) = \sqrt{\frac{40}{\pi}}e^{-40x^2}$.
The numerical schemes, implemented in MATLAB, were excuted on
a desktop PC with 3.6~GHz Intel Core i7-4790 processor and 8GB RAM.
\begin{table}[!hbp]
\centering
\begin{tabular}{c|c|c|c|c}

\hline

\hline
 $J$ &  100    &   200 & 400 & 800          \\
\hline
 CPU times (sec) with direct summation &  15.85   &   50.99 & 179.25 & 671.47  \\
\hline
 CPU times (sec) using the fast algorithm  &  0.13    &   0.16 & 0.34 & 0.79   \\
\hline
\end{tabular}
\caption{Comparison of the CPU times of computing the solution of the FPE with and without the fast algorithm
for the case $\a=0.5, \beta = 0.5, f\equiv 0, d=0, \e=1$ and different $J$'s.}
\label{RunTime}
\end{table}

From the results in Table~\ref{RunTime}, the CPU time for the scheme with the fast algorithm increases as $O(J\log J)$
while the scheme with the direct summation grows quadratically in $J$. This behavior agrees with the theoretical analysis
of the complexity of the algorithms.

\section{Maximum principle}
\setcounter{equation}{0}

In this section, we will show that the semi-discrete scheme \eqref{eq.FPEsd}
satisfies the discrete maximum principle under the condition
the function $m_2(s)$ defined in \eqref{eq.m2} is non-negative for $|s|<1$:
the solution to \eqref{eq.FPEsd}
reaches its maximum and minimum outside the solution domain,
i.e. in $\partial I_{h,t_F}$ defined in
\bear
    I_h = \{j\in \mathbb{Z}:|jh|<1\}, \quad
    I_{h,t_F}=I_h \times (0,t_F],
    \quad
    \partial I_{h,t_F} = \mathbb{Z} \times [0,t_F] \setminus I_{h,t_F},
\enar
where $t_F>0$ is any fixed final time.
We point out the unusual definition of the "boundary" of the solution
domain due to the nonlocal exterior absorbing condition \eqref{eq.ec}.
For weak and strong maximum principles for the original equation \eqref{eq.fpeas} or nonlocal Waldenfels operator, we refer to the work \cite{qiao16}.

\begin{prop} [Maximum principle for the absorbing condition] \label{propMax}

  Assume $m_2(s) \geq 0$ for $s\in (-1,1)$.
  \begin{enumerate}[label=(\roman*)]
   \item\label{prop.1}  If
\begin{equation}
   \frac{{\rm d}V_j}{{\rm d}t} -  (A\mathbf{V})_j \leq 0
   \text{  for } (j,t)\in I_{h,t_F}  \quad \text{ and } \quad V_j = 0
   \text{ for } |j|\geq J,
\label{eq.mp1}
\end{equation}
then
\begin{equation}
	\max_{(j,t)\in \mathbb{Z}\times [0,t_F]} V_j(t) = \max_{(j,t)\in \partial I_{h,t_F}} V_j(t).
\label{eq.mp2}
\end{equation}
    \item\label{prop.2}  If
\begin{equation}
   \frac{{\rm d}V_j}{{\rm d}t} -  (A\mathbf{V})_j \geq 0
   \text{  for } (j,t) \in I_{h,t_F}  \quad \text{ and } \quad V_j = 0
   \text{ for } |j|\geq J,
\label{eq.mp3}
\end{equation}
then
\begin{equation}
	\min_{(j,t)\in \mathbb{Z}\times [0,t_F]} V_j(t) = \min_{(j,t)\in \partial I_{h,t_F}} V_j(t).
\label{eq.mp4}
\end{equation}
\end{enumerate}
\end{prop}

\begin{proof}

\begin{enumerate}
   \item
Let us suppose
\begin{equation}
   \frac{{\rm d}V_j}{{\rm d}t} -  (A\mathbf{V})_j < 0
   \text{  for } (j,t)\in I_{h,t_F}  \quad \text{ and } \quad V_j = 0
   \text{ for } |j|\geq J,
\label{eq.mp5}
\end{equation}
but there exist a point $(j^*,t^*) \in I_{h,t_F}$ such that
$\displaystyle{ V_{j^*}(t^*) = \max_{\mathbb{Z}\times [0,t_F]} V_j}$.

   \item \label{item.p2}
If $0<t^*<t_F$, then
\begin{equation}
	\frac{{\rm d}V_{j^*}}{{\rm d}t} = 0 \text{   at } t=t^*.
\label{eq.cp}
\end{equation}
On the other hand,  we have, at $t=t^*$, $\displaystyle{(s_k-s_{j^*})\frac{V_{j^*}-V_{j^*-1}}{h} \geq  0}$ for $j^*+1\leq k\leq J$
and $\displaystyle{(s_k-s_{j^*}) \frac{V_{j^*+1}-V_{j^*}}{h} \geq 0 }$ for $-J\leq k \leq j^*-1$ because $V_{j^*}(t*)$ is the maximum. Consequently,
each summation term in \eqref{eq.S} is non-positive, resulting
$(S\mathbf{V})_{j^*} \leq 0$ at $t=t^*$. Due to the fact that
$V_{j^*}(t*)$ is the maximum,
$V_{j^*+1} - 2V_{j^*} + V_{j^*-1} \leq 0$,
$m_1(s_{j^*})\delta_u V_{j^*} \geq 0$ from the upwind scheme \eqref{eq.uw},
$m_2(s_{j^*}) V_{j^*} \geq 0$ from the assumption $m_2(s) \geq 0$ and
$V_{j^*} \geq 0 = V_k$ for $k \geq J$, resulting $(B\mathbf{V})_{j^*} \leq 0 $
at $t=t^*$.  Therefore, together with \eqref{eq.cp}, we have
$\displaystyle{\frac{{\rm d}V_{j^*}}{{\rm d}t} - (A\mathbf{V})_{j^*} \geq 0 }$
at $t=t^*$, which is a contradiction to \eqref{eq.mp5}.

\item
   If $t^*=t_F$, then
$\displaystyle{\frac{{\rm d}V_{j^*}}{{\rm d}t} \geq 0}$  at  $t=t^*$ because
$V_{j}$ reaches its maximum at $(j^*,t^*)$ for all $(j,t)\in \mathbb{Z}\times [0,t_F]$.  The argument in the previous point \ref{item.p2} still holds,
i.e., $(A\mathbf{V})_{j^*} \leq 0$ at  $t=t^*$. Thus, it is a contradiction
to \eqref{eq.mp5}.

\item
  In the general case of \eqref{eq.mp1}, we define
$\mathbf{V}^\delta :=\mathbf{V}-  \delta t$
where $\delta$ is a positive parameter.
Then, we have $ A\mathbf{V}^\delta =  A\mathbf{V}$, and
$\displaystyle{
	\frac{{\rm d}V_j^\delta}{{\rm d}t} - (A\mathbf{V}^\delta)_j
	=  \frac{{\rm d}V_j}{{\rm d}t} -\delta - (A\mathbf{V})_j <0
}$ for $(j,t) \in I_{h,t_F}$.
Thus, we have $ \max_{(j,t)\in \mathbb{Z}\times [0,t_F]} V_j^\delta(t) = \max_{(j,t)\in \partial I_{h,t_F}} V_j^\delta(t)$.
We obtain \eqref{eq.mp2} by letting $\delta \rightarrow 0$.
This concludes the proof of the assertion \ref{prop.1}.

\item By considering $-\mathbf{V}$, the assertion \ref{prop.2} follows immediately.

\end{enumerate}
\end{proof}

\begin{remark} \label{Rem1}
  The condition $m_2\geq 0$ for the maximum principle
is equivalent to  requiring the drift $f$ satisfy
$f' \geq -\min_{s \in (-1,1)}\frac{\varepsilon b^{-\a}}{\a}[\frac{C_n}{(1-s)^\a}+\frac{C_p}{(1+s)^\a}]$.
For example, if $\a=\beta=0.5, \e=1, b=1$, the RHS of the above inequality
is $-0.76$, therefore, the O-U potential $f = -0.6x$ would satisfy the maximum principle. We point out that any drift $f$ with $f' \geq 0$ satisfies
the condition $m_2\geq 0$.
\end{remark}

\begin{prop} [Maximum principle for the natural condition] \label{propMaxNat}

  Assume $f'(bs) \geq 0$ for $s\in (-1,1)$.
  \begin{enumerate}[label=(\roman*)]
   \item\label{prop.Nat1}  If
\begin{equation}
   \frac{{\rm d}V_j}{{\rm d}t} -  (A\mathbf{V})_j \leq 0
   \text{  for } (j,t)\in I_{h,t_F} ,
\label{eq.Natmp1}
\end{equation}
then
\begin{equation}
	\max_{(j,t)\in \mathbb{Z}\times [0,t_F]} V_j(t) = \max_{(j,t)\in \partial I_{h,t_F}} V_j(t).
\label{eq.Natmp2}
\end{equation}
    \item\label{prop.Nat2}  If
\begin{equation}
   \frac{{\rm d}V_j}{{\rm d}t} -  (A\mathbf{V})_j \geq 0
   \text{  for } (j,t) \in I_{h,t_F}
\label{eq.Natmp3}
\end{equation}
then
\begin{equation}
	\min_{(j,t)\in \mathbb{Z}\times [0,t_F]} V_j(t) = \min_{(j,t)\in \partial I_{h,t_F}} V_j(t).
\label{eq.Natmp4}
\end{equation}
\end{enumerate}
\end{prop}

\begin{proof}
  We note that the numerical scheme for the natural condition is the same
as that for the absorbing condition except the definition of $m_2$ is
simplified to $c'(bs)$ or $f'(bs)$ (see Remark~\ref{rem.nc} and Eq.~\eqref{eq.cs}).
Consequently, the proof is the same as that for Proposition~\ref{propMax}.
\end{proof}

\begin{remark} \label{Rem0}
  One can find many fully discrete schemes that satisfies maximum principle.
For example, one can show
that, if  $ m_2(s) \geq 0$ for $s \in (-1,1)$,
then the solution to the backward Euler scheme for the time integration
\bear \label{ImplicitEuler}
\frac{V_j^n-V_j^{n-1}}{\Delta t} = (A\mathbf{V}^n)_j, \quad \text{ for }
   -J+1\leq j\leq J-1, \quad n=1, 2, \cdots
\enar
where $V_j^n$ means the numerical solution of $v$ at $(s_j,t_n)$
with $t_n=n\Delta t$ and $\mathbf{V}^n= (V_{-J+1}^n, \cdots, V_{J-1}^n)^T$,
satisfies the maximum principle
\bear
\max_{(j,n) \in I_{h}\times I_n} V_j^n \leq \max_{(j,n) \in \mathbb{Z} \times
\bar{I_n} \setminus I_h\times I_n } V_{j}^n, \quad
\min_{(j,n) \in I_{h}\times I_n} V_j^n \geq \min_{(j,n) \in \mathbb{Z} \times
\bar{I_n} \setminus I_h\times I_n } V_{j}^n,
\enar
where $\Delta t=t_F/n_F$ for some total number of time steps
$n_F \in \mathbb{Z}$ and
\begin{equation}
	I_n = \{ n\in \mathbb{Z}: 1\leq n \leq n_F\}, \quad
	\bar{I_n} = \{ n\in \mathbb{Z}: 0\leq n \leq n_F\}.
\end{equation}
The proof will follows closely with that for Proposition~\ref{propMax}
by realizing that, if $V_j^n$ reaches the maximum at $V_{j^*}^{n^*}$, then
\bear
\frac{V_{j^*}^{n^*}-V_{j^*}^{n^*-1}}{\Delta t} \geq 0.
\enar
\end{remark}

In the following, we mainly use the backward Euler for time evolution as
the probability density will be negative near the boundary for the forward Euler sheme. The
Fig.~\ref{ComUpWNorm1} shows the situation for $\a=0.5, \beta=0.5, f=-x, \s=0, \e=1$.

\befig[h]
 \begin{center}
\includegraphics[width=0.8\linewidth]{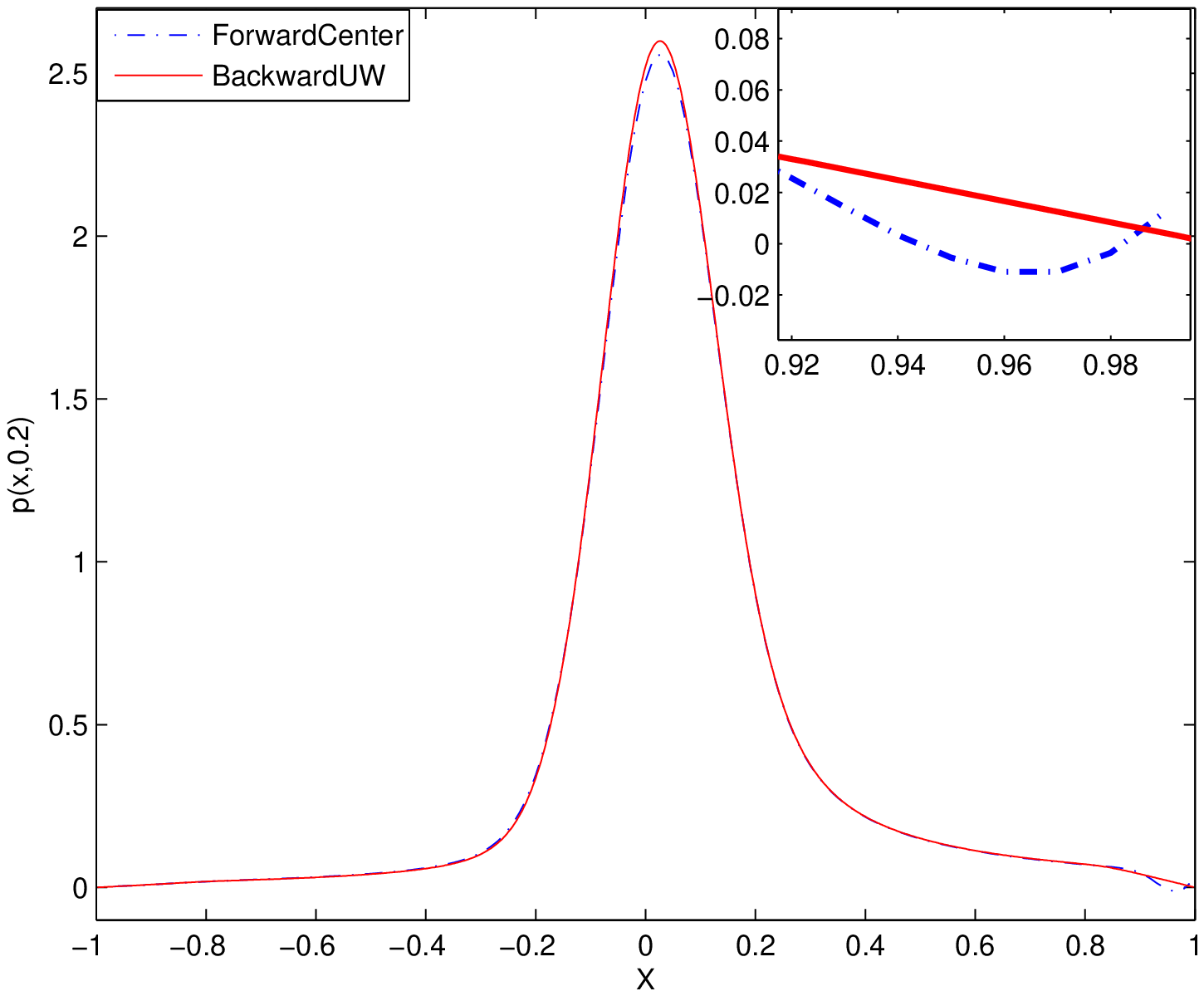}
%\vspace{-0.3in}
\end{center}
\caption{Different numerical scheme for FPE with $\a=0.5, \beta=0.5, f=-x, \s=0, \e=1$. The blue solid
line provide the numerical solution for time evolution with forward Euler(blue solid line) and backward Euler(red dash-dot line).  }
\label{ComUpWNorm1}
\enfig

\begin{remark} \label{rem.conv}
   The stability of the schemes based on (\ref{eq.FPEsd}) follows immediately
from the maximum principle we have proved in the section. It is obvious
that the fully discretized schemes are consistent, where the details of the truncation error is discussed in \cite{Xiao2016}. Since the equations ~(\ref{FPKchange1}) and (\ref{FPKchange2}) are linear, the schemes are convergent due to
the Lax Equivalence theorem.
\end{remark}

\section{Numerical results}
\subsection{Verfication}
\setcounter{equation}{0}

Before discussing the evolution of the PDFs obtained from the simulations,
we validate our numerical methods by comparing with a known exact solution.
Based on the density function for L\'evy distribution($\a=0.5, \b=1$)
and the scaling property of stable random variables,
we have the PDF for $L_t$
\cite{Apple, Duanbook2, Taqqu}
\bear \label{LevyDis}
    p(x,t) = \frac{x^{-\frac32}t}{\sqrt{2\pi}} e^{-\frac{t^2}{2x}},\quad \text{for} \;  x>0; \qquad p(x,t) = 0,\quad \text{for} \;  x\leq 0 .
\enar

%In general, the analytic solutions of the corresponding Fokker-Planck equation are difficult to obtain.
%But for some special cases, the probability density function $p(x,t)$ can be represented as infinite series\cite{Taqqu}.
%By the characteristic functions of $\a$-stable random variables, we can also take the Fourier transform to obtain the densities
%and use fast Fourier transform (FFT) to compute the numerical integral\cite{SimuSaS,Nolan}. There are closed forms for several important
%cases, such as normal distribution($\a = 2, \b=0$), Cauchy distribution($\a = 1, \b = 0$) and L\'evy distribution($\a=0.5, \b=1$) and so on.
%The density of L\'evy distribution is $ f_X(x) = \frac{x^{-\frac32}}{\sqrt{2\pi}}e^{-\frac{1}{2x}}$ for $x>0$. Using the property of stable
%random variables, if $X_1\sim S_\a(1,\b,0)$, then $X_t \sim S_\a(|t|^{\frac{1}{\a}}, \text{sign}(t^{\frac{1}{\a}})\b,0)$.

\befig[h]
\begin{center}
\includegraphics[width=0.8\linewidth]{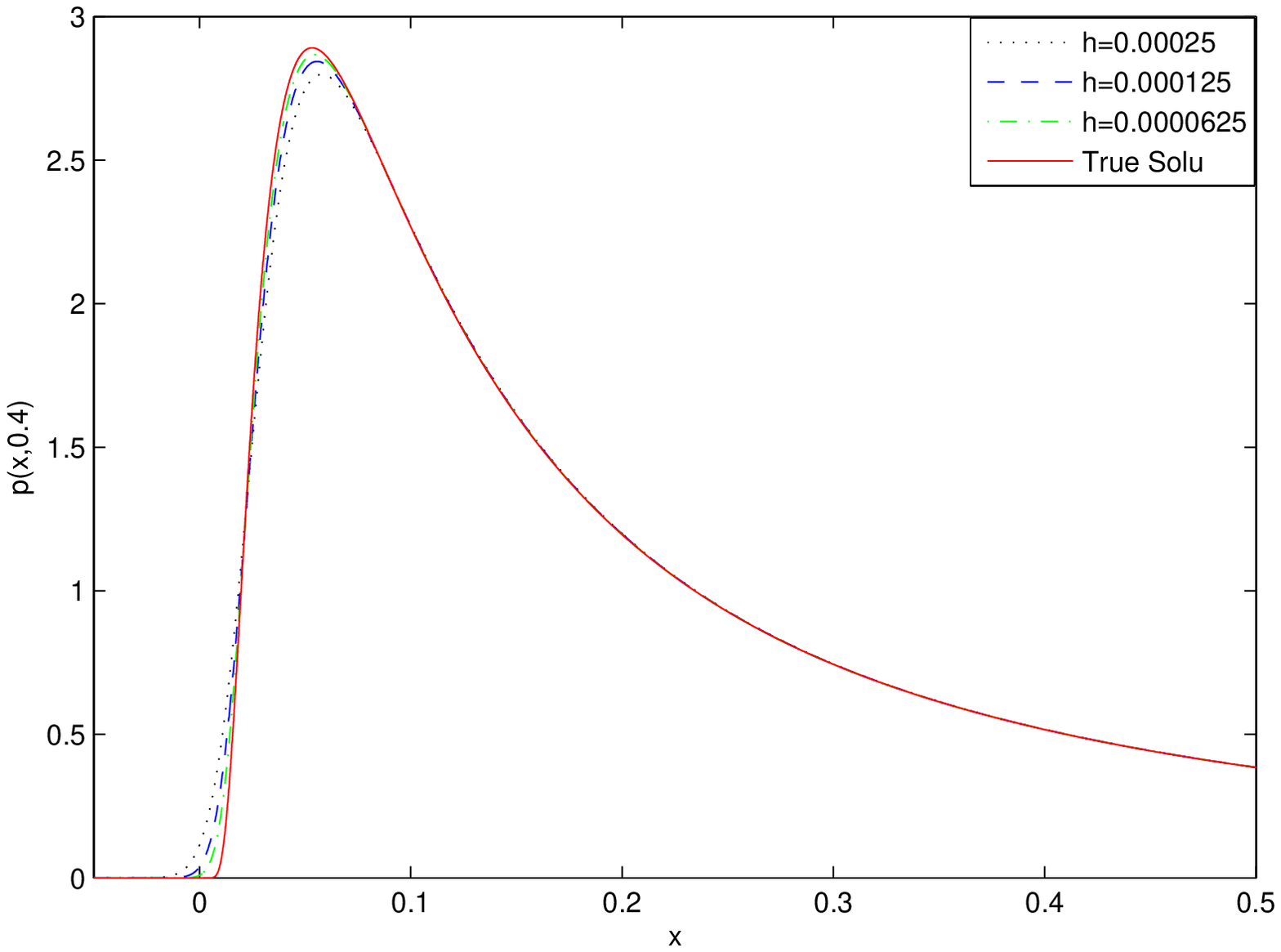}
\vspace{-0.3in}
\end{center}
\caption{ Comparison between the analytic solution (the solid line) and
the numerical solutions at time
$t_F=0.4$ for $\a=0.5, \beta=1, f\equiv 0, \sigma=0, \e=1$ and
different resolutions $h=0.00025$ (the dotted line),
$0.000125$ (the dashed line) and $0.0000625$ (the dash-dotted line).}
\label{fig.cklevy}
\enfig
To compare our numerical solution with the exact solution~(\ref{LevyDis}),
we start our numerical computation from the time $t=0.2$
by setting the initial condition to be
$ p(x,0.2)$ given in \eqref{LevyDis}.
Noticing the analytic solution correspond to the natural condition,
we take the computational domain $D=(-10,10)$ and $\a=0.5, \beta=1, f\equiv 0,
\sigma=0, \e=1$.
Figure~\ref{fig.cklevy} shows the solutions at the time $t_F=0.4$
with different spatial resolutions $h=0.00025, 0.000125, 0.0000625$
and the time step-size $\Delta t=0.5h$.
The results show that the numerical solutions agree with the exact solution
well and the difference or the error decreases as we increase the resolution.

\subsection{Evolution of PDFs}
\befig[h]
 \begin{center}
\includegraphics[width=0.8\linewidth]{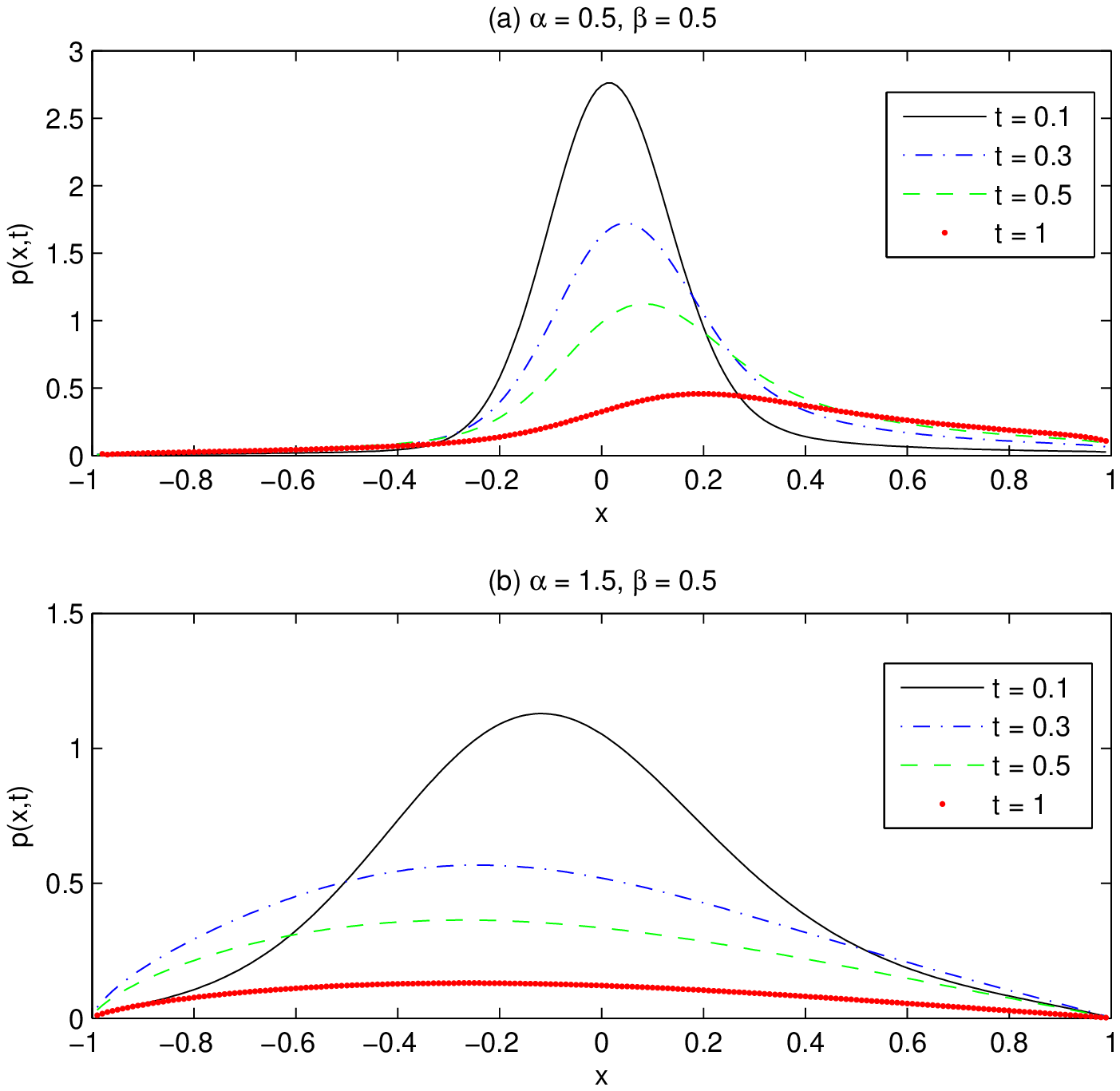}
\vspace{-0.3in}
\end{center}
\caption{ The evolution of PDFs $p$ for $\sigma=0, f\equiv  0,\b=0.5, \e=1$
subject to the absorbing condition with $D=(-1,1)$ and
the initial condition $p(x,0) = \sqrt{\frac{40}{\pi}}e^{-40x^2}$.
(a) $\a=0.5$; (b) $\a=1.5$. }
\label{DifT}
\enfig
One of the ways to understand the behavior of a stochastic process governed
by the SDE \eqref{SDE01} is through its PDF. First, we numerically find
the PDF corresponding to $L_t$ subject to the absorbing condition
with $D=(-1,1)$. Initially, the location of the process is at $x$
has the probability
$p(x,0) = \sqrt{\frac{40}{\pi}}e^{-40x^2}$,
having a sharp peak at the origin $0$.
The time evolution of the probability density $p$ is shown in Fig.~\ref{DifT}
for $f\equiv0, \sigma=0, \beta=0.5, \e=1$ and two different values
of $\a=0.5, 1.5$. For $\a=0.5$, the peak decays and moves to the right;
for $\a=1.5$, the peak decays faster but moves to the left.
Because the skewness parameter $\beta$ is positive, there is larger tendency
to jump to the right. The linear drift coefficient $K_{\a,\beta}$
due to the compensation is positive when $\a<1$ and becomes negative
when $\a>1$. The jump direction and the drift work together to render the movement of the peak in the case of $\a=0.5$. On the other hand, when $\a=1.5$,
the two factors compete and the effect of the drift dominates.
Further more, we note that the shape of the PDF becomes convex
after $t\geq 0.3$ and more smooth for $\a=1.5$,  while the PDF for $\a=0.5$
appears to be discontinuous at the right boundary $x=1$ at large times.

%\subsection{The effect of parameters}

Next, we investigate the effects of different parameters on the solution
to the FPE~\eqref{eq.fpeas}, including the skewness parameter $\b$,
the drift $f$, the intensities of Gaussian noise $\s$ and non-Gaussian noise $\e$, the domain size $D$ and
the auxiliary conditions.

\befig[h]
 \begin{center}
\includegraphics[width=0.8\linewidth]{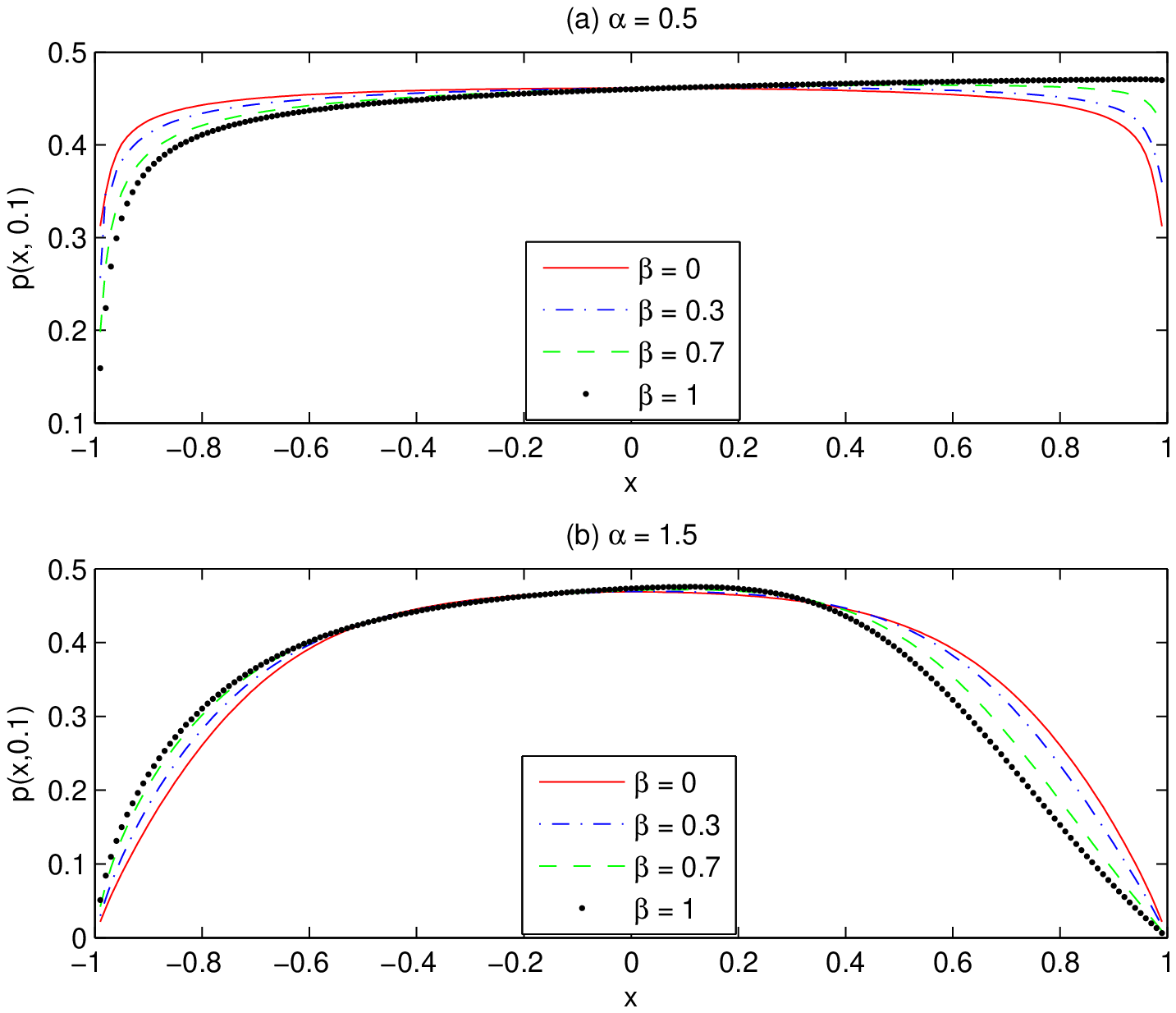}
\vspace{-0.3in}
\end{center}
\caption{The effect of the skewness parameter $\beta$. (a) The PDFs at time
$t=0.1$ are plotted for different values of $\b=0, 0.3, 0.7, 1$ with
$\a = 0.5, f \equiv 0, \s=0, \e=1, D=(-1,1)$ and
the initial condition $p(x,0) = 0.5I_{\{|x|<1\}}$.
  (b) The same as (a) except $\a=1.5$. }
\label{DifBeta}
\enfig

\befig[h]
 \begin{center}
\includegraphics[width=0.8\linewidth]{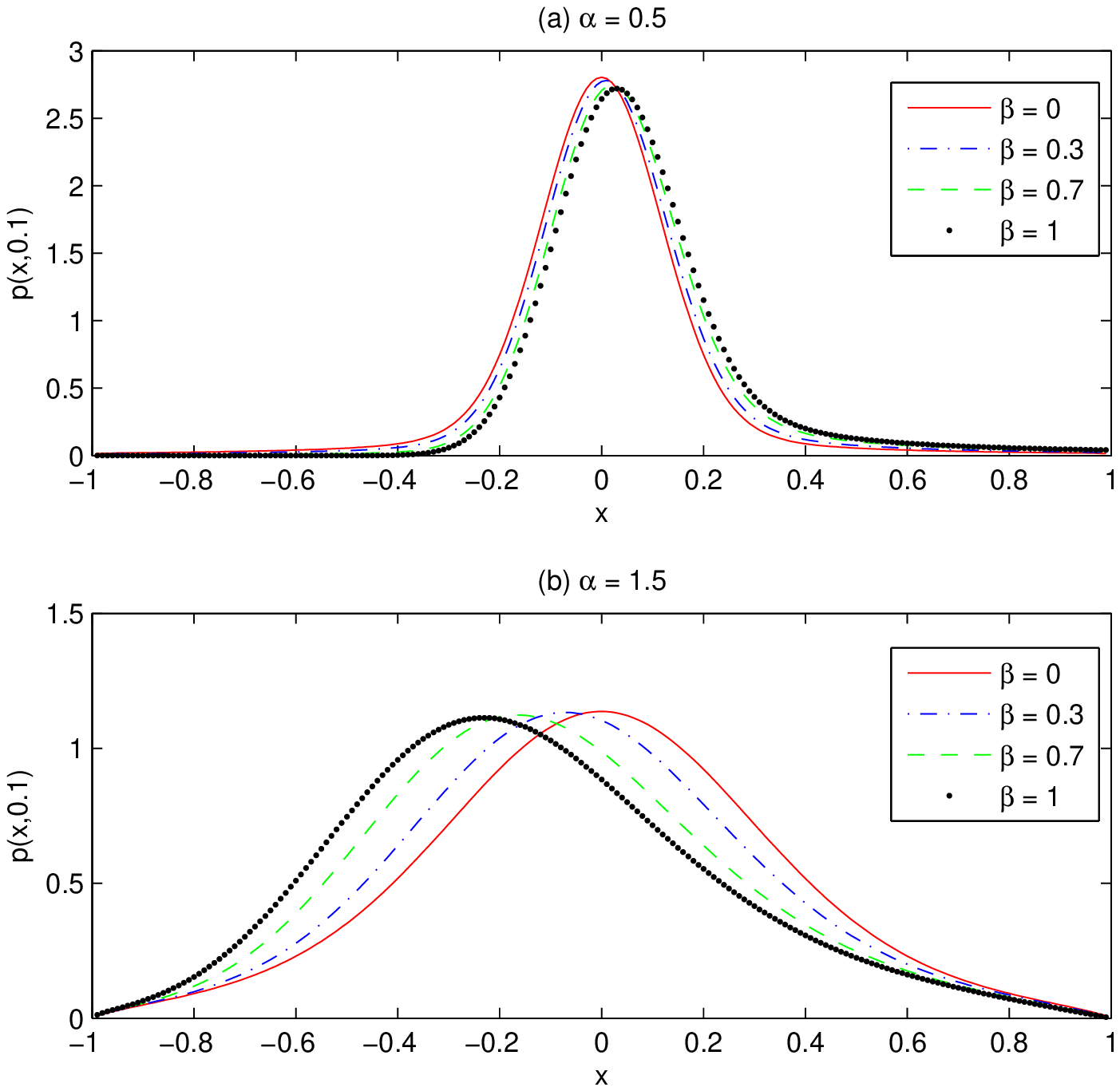}
\vspace{-0.3in}
\end{center}
\caption{The effect of the skewness parameter $\beta$ on the PDFs at time
$t=0.1$ for $f\equiv 0, \s=0, \e=1, D=(-1,1)$ and the initial condition $p(x,0) = \sqrt{\frac{40}{\pi}}e^{-40x^2}$.
(a) $\a=0.5$; (b) $\a=1.5$.  }
\label{DifBetaExp}
\enfig
We consider the effect of the skewness parameter $\beta$ on the PDFs with
initial condition of a uniform distribution $p(x,0) = 0.5I_{\{|x|<1\}}$
in Fig.~\ref{DifBeta}. The impact of $\beta$ on the PDF is
different for $0<\a<1$ and $1<\a<2$.
For $\a=0.5$ (Fig.~\ref{DifBeta}(a)),
the PDFs have relatively flat profiles in the middle but drop to zero sharply
at the left boundary of the domain $x=-1$. As $\b$ increases,
the probability profile tilts toward the right, i.e., having larger
probability for positive $x$. We point out the interesting behavior
of the PDFs at the right boundary of the domain, where $p(x,0.1)$
becomes increasingly discontinuous as $\b$ approaches $1$.
On the other hand, for
$\a=1.5$ (Fig.~\ref{DifBeta}(b)),
the profiles of the PDFs are much smoother than those of $\a <1$
and the values of the PDFs are larger near the left boundary
as $\b$ increases. It is interesting to note that the PDFs for $\b>0$
reach their maxima at small positive $x$ values.
Similar effects of $\b$ can be seen from another example shown
in Fig.~\ref{DifBetaExp} where the initial condition
is a Gaussian distribution $p(x,0) = \sqrt{\frac{40}{\pi}}e^{-40x^2}$.
As $\b$ increases, the hump in the PDF profile
shifts to the right for $\a=0.5$ and to the left for $\a=1.5$.
Notice that the maximum of the PDF decreases slightly
as $\b$ is raised from $0$ in both cases.

\befig[h]
 \begin{center}
\includegraphics[width=0.8\linewidth]{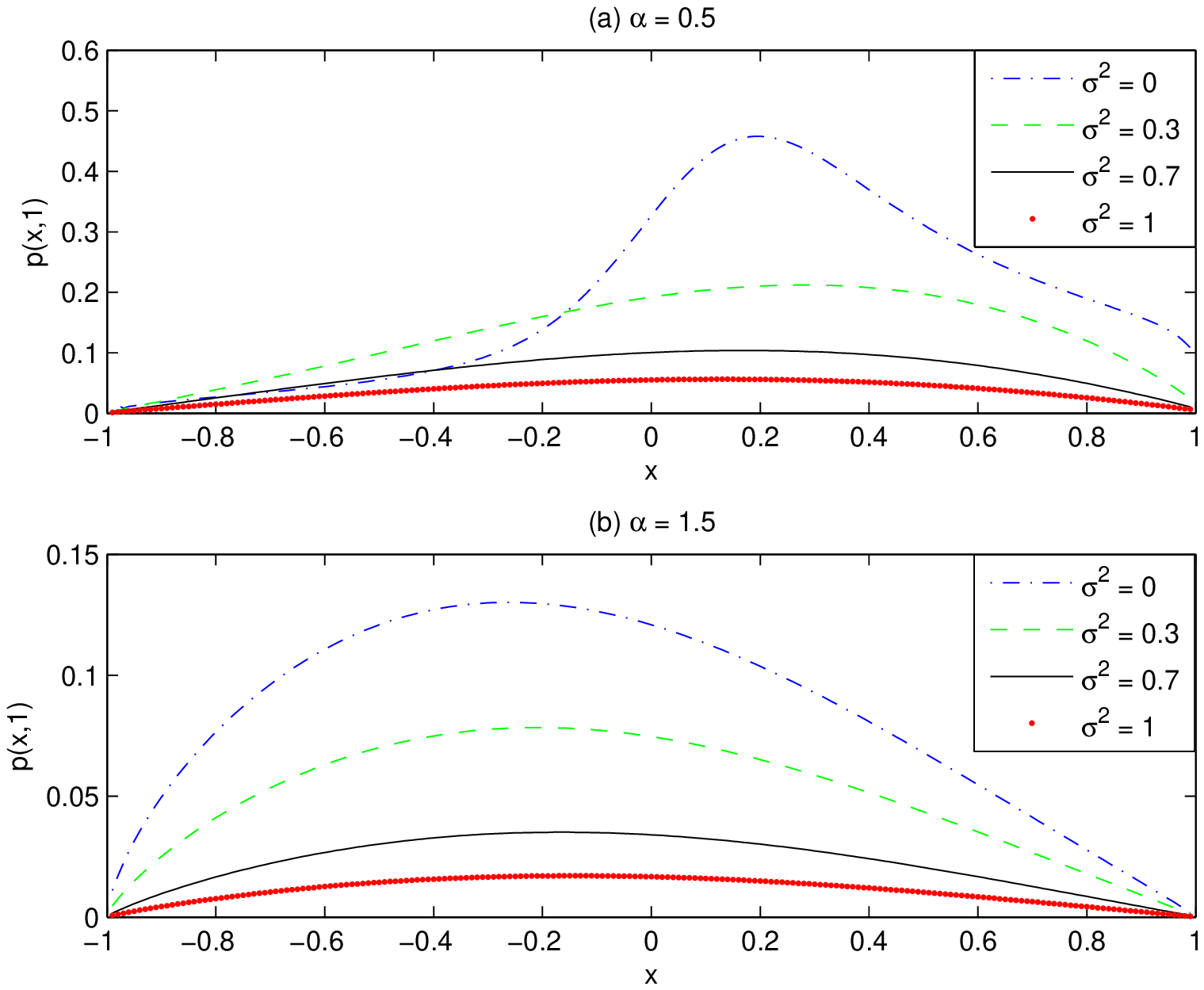}
\vspace{-0.3in}
\end{center}
\caption{
The effect of the intensity of Gaussian noise ($\s^2=0,0.3,0.7,1$) on the PDFs without drift ($f\equiv 0$) at the fixed time
$t=1$ with $\beta=0.5,\e=1, D=(-1,1)$ and the initial condition $p(x,0) = \sqrt{\frac{40}{\pi}}e^{-40x^2}$.
(a) $\a=0.5$;(b) $\a=1.5$. }
\label{ChangeDif}
\enfig

Figure~\ref{ChangeDif} shows the dependence of the probability density $p$
at time $t=1$ on the intensity of Gaussian noise $\s$, where the initial profile
is the Gaussian $p(x,0) = \sqrt{\frac{40}{\pi}}e^{-40x^2}$ and
the other parameters are fixed at $f\equiv 0, \beta=0.5,\e=1$ and $D=(-1,1)$.
Clearly, as one increases the amount of Gaussian noise $\s$, the process is
less likely to stay in the domain $D$ and the values of the PDFs
become smaller. Besides, the profiles of the PDFs at $t=1$
are all concave downward and are smooth at the boundaries of the domain
$x=\pm 1$ for $\s>0$. The graphs of the PDFs become more symmetric
with respect to the center of the domain when there are more Gaussian noises
or $\s$ increases.

\befig[h]
 \begin{center}
\includegraphics[width=0.8\linewidth]{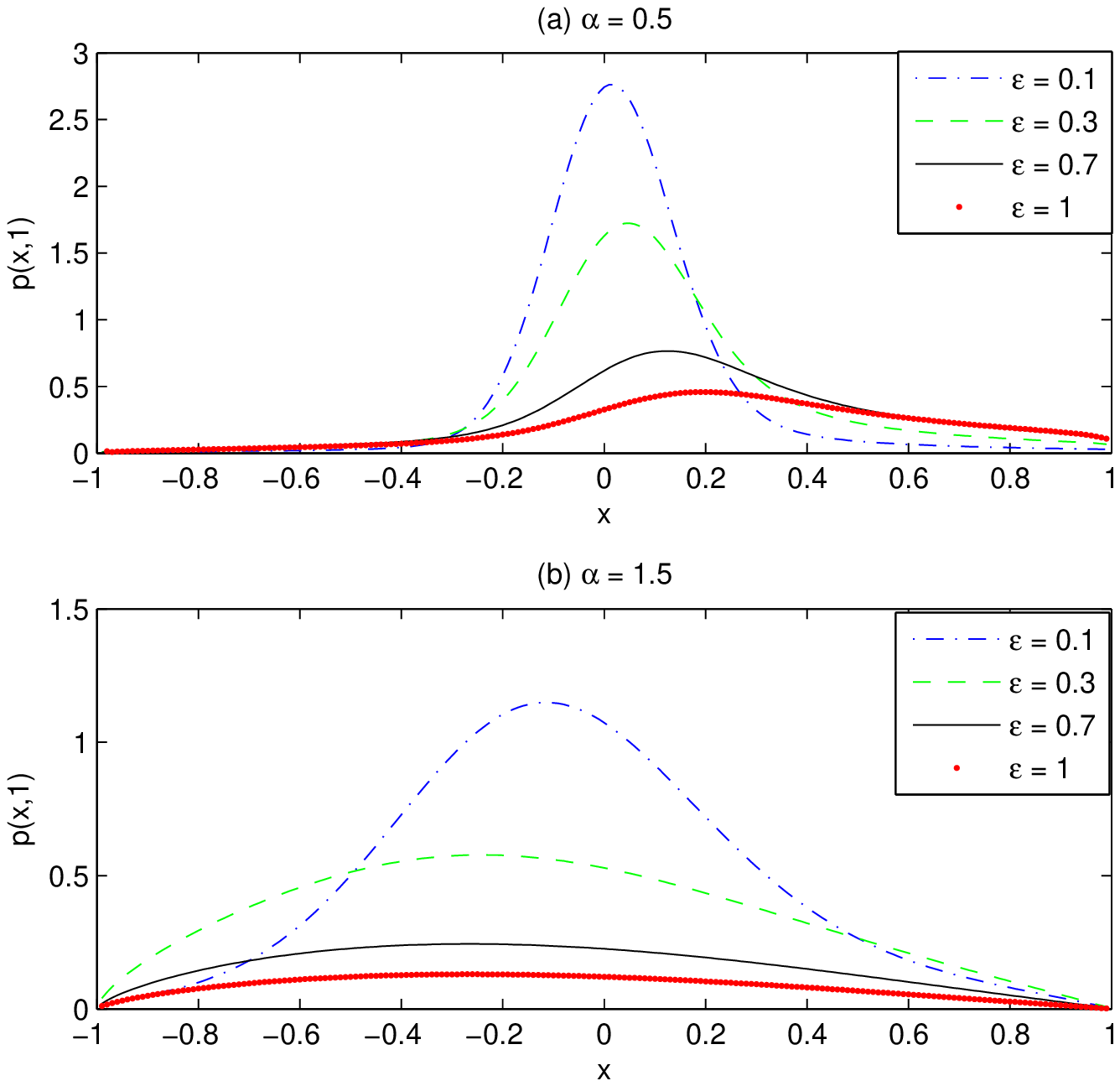}
\vspace{-0.3in}
\end{center}
\caption{ The effect of intensity of non-Gaussian noise ($\e=0.1,0.3,0.7,1$)
on the PDFs without the drift term ($f\equiv 0$) at time
$t=1$ with $\s=0, \beta=0.5,D=(-1,1)$ and the initial condition $p(x,0) = \sqrt{\frac{40}{\pi}}e^{-40x^2}$.
(a) $\a=0.5$; (b)$\a=1.5$. }
\label{ChangeEps}
\enfig
Keeping other parameters and the conditions as in Fig.~\ref{ChangeDif},
we examine the effect of the magnitude of non-Gaussian ($\e$) noises in
Fig.~\ref{ChangeEps} using different values of $\e=0.1, 0.3, 0.7$ and $1$.
The numerical results show that, for $\a=0.5$, the graphs of $p$ develop
a jump at the right boundary of $D$ and become more skew to the right
as $\e$ is increased; for $\a=1.5$, the PDFs become more skew to the left
when the non-Gaussian noise level is raised.

\befig[h]
 \begin{center}
\includegraphics[width=0.8\linewidth]{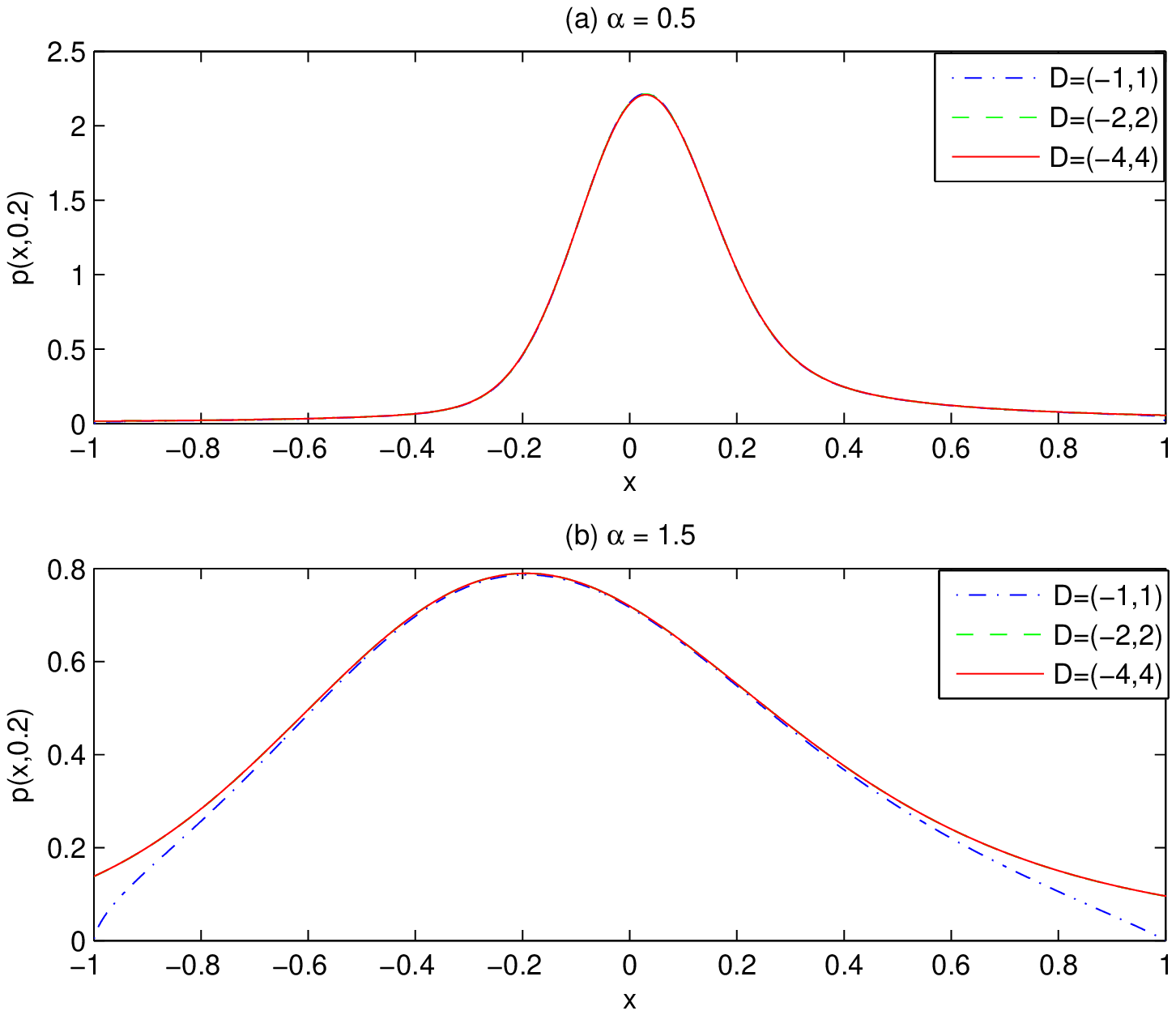}
\vspace{-0.3in}
\end{center}
\caption{The effect of the domain size $D$($D=(-1,1),(-2,2),(-4,4)$) for the absorbing condition with
$\beta=0.5,\s=0,\e=1,f\equiv 0$, at $t=0.2$.
The initial condition is the Gaussian density function
$p(x,0) = \sqrt{\frac{40}{\pi}}e^{-40x^2}$.
(a) $\a=0.5$;(b) $\a=1.5$.  }
\label{DiffDom}
\enfig
Figure~\ref{DiffDom} shows that the variety of the densities at time $t=0.2$ for different domains $D=(-1,1),(-2,2),(-4,4)$
and different $\a=0.5,1.5$ starting with the same Gaussian-type initial condition.
The auxiliary condition is the absorbing condition and $\beta=0.5, \e=1, \s=0$
without the drift.
From Fig.~\ref{DiffDom}(a) corresponding to $\a=0.5$,
we find that the densities for the differently sized domains are almost
identical on the interval $(-1,1)$. It can be explained by realizing
that most of the movement of the process governed by the SDE \eqref{SDE01}
consists of large jumps for $\a=0.5$.
For $\a=1.5$ as shown in Fig.~\ref{DiffDom}(b),
the probability finding the process near the peak $(-0.4,0)$ is about the same
for all three sizes of the domain $D$ but the density for the smallest domain
$D=(-1,1)$ quickly goes to zero at the boundary $x=\pm 1$.
It is interesting to note that the PDFs for the two larger domains $D=(-2,2)$
and $(-4,4)$ are almost identical on the interval $(-1,1)$.

\befig[h]
\begin{center}
\includegraphics[width=0.8\linewidth]{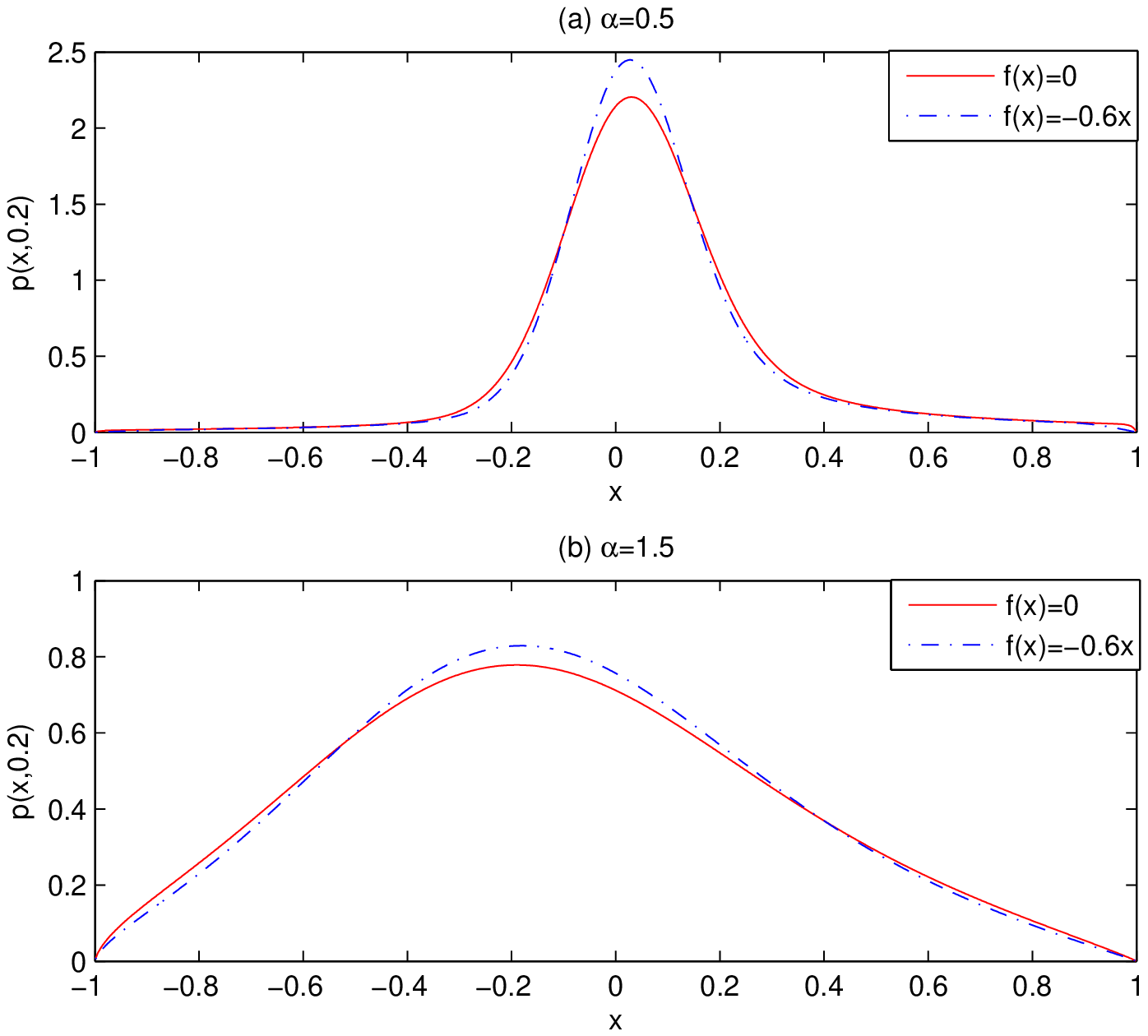}
\vspace{-0.3in}
\end{center}
\caption{ The effect of the drift $f$($f(x)= 0$ or $-0.6x$) on
the PDFs at time $t=0.2$ with $\beta=0.5,\s=0,\e=1, D=(-1,1)$
and the initial condition $p(x,0) = \sqrt{\frac{40}{\pi}}e^{-40x^2}$ for $\a=0.5$(a), $\a=1.5$(b).
 }
\label{DifDrift}
\enfig
Further, we investigate the effect of drift $f$. Figure~\ref{DifDrift} shows the changes of probability densities
if we add the O-U potential $f(x)=-0.6x$ for different $\a=0.5$(Fig.~\ref{DifDrift}(a)) and $\a=15$(Fig.~\ref{DifDrift}(b))
with $\beta=0.5,\s=0,\e=1, D=(-1,1)$ starting with the initial
profile of $p(x,0) = \sqrt{\frac{40}{\pi}}e^{-40x^2}$.
It can be seen that the densities become larger near the origin and
more symmetric with respect to the center of the domain as expected.

\befig[h]
 \begin{center}
\includegraphics[width=0.8\linewidth]{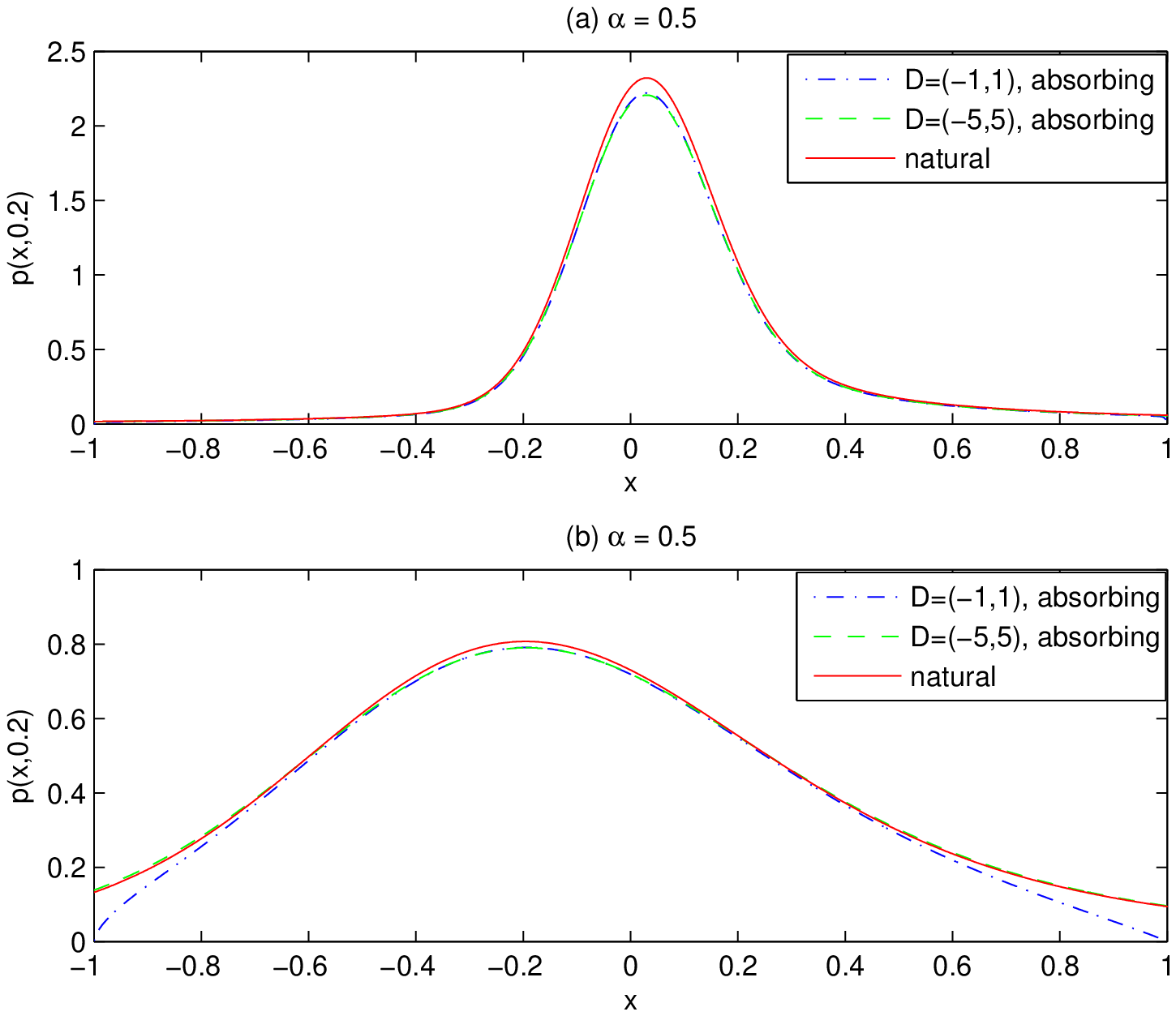}
\vspace{-0.3in}
\end{center}
\caption{
Comparison of the PDFs at time $t=0.2$ between
the natural condition and the absorbing condition for the domains $D=(-1,1)$
and $(-5,5)$ with $\beta=0.5,\s=0,\e=1,f\equiv0$
and the initial condition of $p(x,0) = \sqrt{\frac{40}{\pi}}e^{-40x^2}$.
(a) $\a=0.5$; (b) $\a=1.5$.  }
\label{CompBoundary}
\enfig
Finally, we consider the effect of auxiliary conditions described in \S~\ref{sec.ac}. In Fig.~\ref{CompBoundary}, we plot the PDFs at the same time $t=0.2$
but with three different auxiliary conditions: the natural condition
and the absorbing conditions for two domain sizes $D=(-1,1)$ and $(-5,5)$.
We keep the other parameters the same as those in Fig.~\ref{DiffDom}.
For both values of $\a=0.5$ and $1.5$,
the PDF for the natural condition is slightly larger
than the ones for the absorbing condition. For $\a=0.5$,
the PDFs for the three different auxiliary conditions
are almost identical on the interval $(-1,1)$ except near the peak region.
For $\a=1.5$, the PDF for the absorbing condition with $D=(-1,1)$ drops
to zero near the boundary of its domain $x=\pm 1$ while
the PDF for the absorbing condition with $D=(-5,5)$ is close to that
of the natural condition on the interval $(-1,1)$.

\subsection{Most probable phase portrait (MPPP)}

\befig[h]
 \begin{center}
\includegraphics[width=0.8\linewidth]{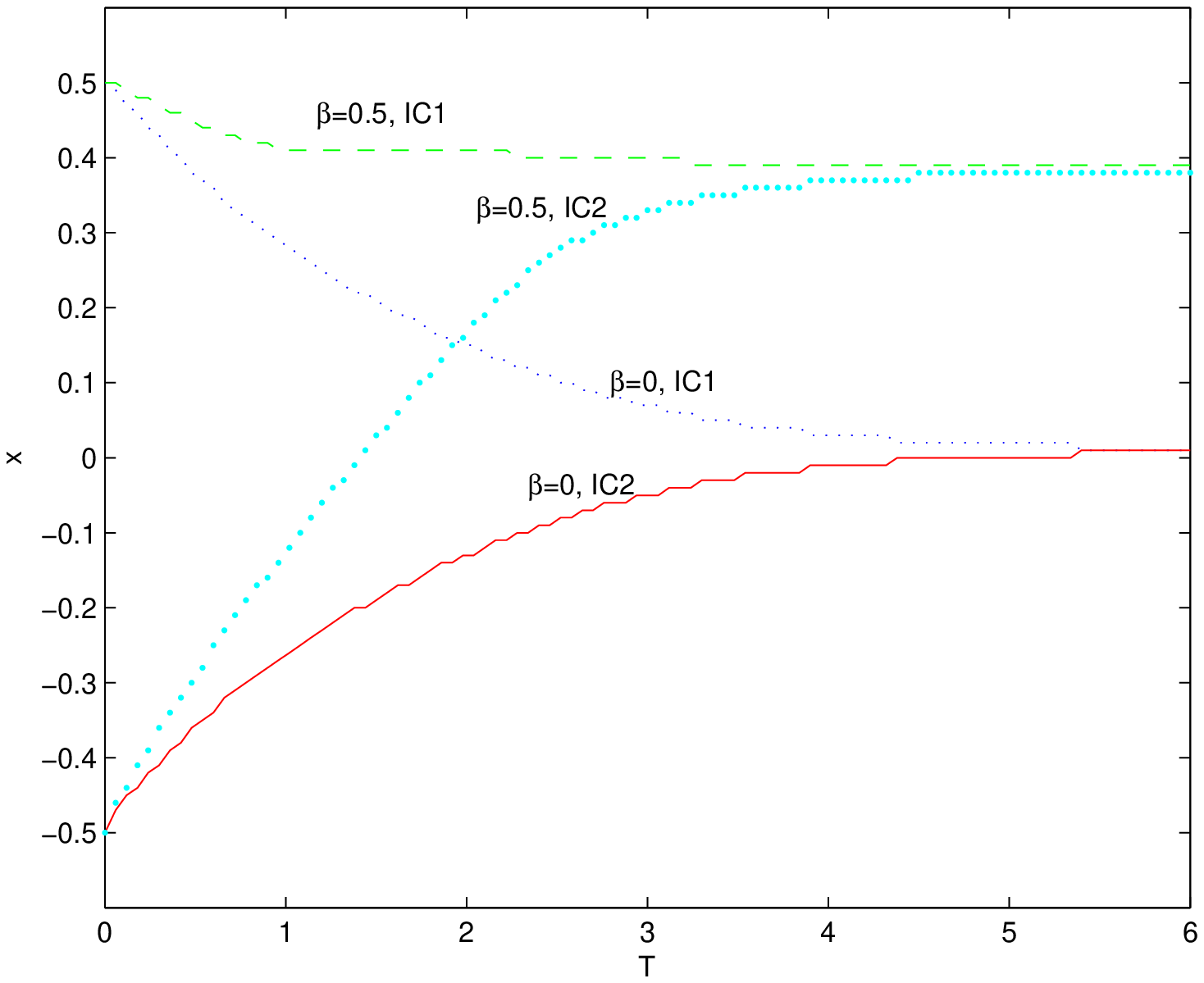}
%\vspace{-0.3in}
\end{center}
\caption{
The MPPP for different $\beta=0, 0.5$ with the drift term $f(x)=-0.6x$ and $\a=0.5,d=0,\e=1, D=(-1,1)$.
The initial condition 1(IC1) and 2(IC2) correspond to the Gaussian
density functions $p(x,0) = \sqrt{\frac{40}{\pi}}e^{-40(x-0.5)^2}$
and $p(x,0) = \sqrt{\frac{40}{\pi}}e^{-40(x+0.5)^2}$, respectively.  }
\label{MPPPx}
\enfig
From the solutions to the Fokker-Planck equation \eqref{eq.fpeas},
one can find the most probable orbits of stochastic dynamical
systems \eqref{SDE01} driven by asymmetric L\'evy motion.
MPPP, denoted by $x_m(t)$, is defined as the maximum of the PDF at time $t$,
i.e. $x_m(t)=\max_{x \in \R}p(x,t)$,
which gives the most probable orbit starting at $x_0$ \cite{Duanbook2}.
Figure~\ref{MPPPx} plots the MPPPs for $\a=0.5, \s=0, \e=1, D=(-1,1), f(x)=-0.6x$
and different values of $\beta=-0.5, 0, 0.5$.
To approximate the delta function $\delta(x_0)$,
we choose the initial condition of Gaussian density function $p(x,0) = \sqrt{\frac{40}{\pi}}e^{-40(x-x_0)^2}$. In absence of the noises, the system
has the unique globally stable state at the origin $x=0$. When we introduce
the symmetric L\'evy noises corresponding to $\beta=0$, the process
still goes to the origin independent of the initial starting point $x_0$.
When the asymmetric L\'evy noises are present ($\beta=0.5$),
the MPPP approaches to a point different than the origin.
It is interesting that MPPP converges to the same point for a fixed value
of $\beta$ regardless the initial condition.

\section{Conclusion}

Due to its wide range of applications of non-Gaussian L\'evy noises in many
disciplines, we study the
Fokker-Planck equation with asymmetric $\a$-stable L\'evy motion,
which is a nonlocal (integro-differential) partial differential equation.
The Fokker-Planck equation describes the time evolution of the probability density function.
In this work, we show a symmetry property for solutions
with respect to the sign of $\b$, enabling us only need to consider the cases
with $\b>0$. We have developed an accurate and fast numerical scheme
for solving the FPEs for different auxiliary conditions
(the absorbing condition and the natural condition).
The numerical method is validated by comparing the numerical solution
with a special exact solution and used to compute the solutions corresponding
to different parameters in the system.  We find that
the PDFs are discontinuous at the right boundary when $\a<1$ and
$\b>0$ and the discontinuity becomes more evident when $\b$ increases;
the discontinuity disappears for $\a>1$.
We have also considered the most probable phase portrait and find
that the process approaches the same state when starting with different
condition.

\section{Acknowledgements}
The research is partially supported by the grants China Scholarship Council \#201306160071 (X.W.), NSF-DMS \#1620449 (J.D. and X.L.),
National Natural Science Foundation of China grants 11531006 and 11771449(J.D.).


\begin{thebibliography}{10}

\bibitem{Acosta17}
G.~Acosta and J.~Borthagaray.
\newblock A fractional {L}aplace equation: Regularity of solutions and finite
  element approximations.
\newblock {\em SIAM J. Numer. Anal.}, 55(2):472--495, 2017.

\bibitem{AcostaMC2017}
G.~Acosta, J.~Borthagaray, O.~Bruno, and M.~Maas.
\newblock Regularity theory and high order numerical methods for the
  (1{D})-fractional {L}aplacian.
\newblock {\em Math. Comput.}, In press, 2018.

\bibitem{Apple}
D.~Applebaum.
\newblock {\em {L}{\'e}vy Processes and Stochastic Calculus}.
\newblock Cambridge University Press, Cambridge, U.K., 2nd edition, 2009.

\bibitem{Cartea07}
A.~Cartea and D.~del Castillo-Negrete.
\newblock Fractional diffusion models of option prices in markets with jumps.
\newblock {\em Phys. A, Stat. Mech. Appl.}, 374(2):749--763, 2007.

\bibitem{CHEN17}
Z.~Chen, E.~Hu, L.~Xie, and X.~Zhang.
\newblock Heat kernels for non-symmetric diffusion operators with jumps.
\newblock {\em J. Differ. Equations}, 263(10):6576--6634, 2017.

\bibitem{Chen2016}
Z.~Chen and X.~Zhang.
\newblock Heat kernels and analyticity of non-symmetric jump diffusion
  semigroups.
\newblock {\em Probab. Theory Rel.}, 165(1):267--312, 2016.

\bibitem{Cozzi17}
M.~Cozzi.
\newblock Interior regularity of solutions of non-local equations in {S}obolev
  and {N}ikol'skii spaces.
\newblock {\em Ann. Mat. Pura. Appl.}, 196(2):555--578, 2017.

\bibitem{DELIA13}
M.~D'Elia and M.~Gunzburger.
\newblock The fractional {L}aplacian operator on bounded domains as a special
  case of the nonlocal diffusion operator.
\newblock {\em Comput. Math. Appl.}, 66(7):1245--1260, 2013.

\bibitem{Du2012}
Q.~Du, M.~Gunzburger, R.~B. Lehoucq, and K.~Zhou.
\newblock Analysis and approximation of nonlocal diffusion problems with volume
  constraints.
\newblock {\em SIAM Rev.}, 54(4):667--696, 2012.

\bibitem{Duanbook2}
J.~Duan.
\newblock {\em An Introduction to Stochastic Dynamics}.
\newblock Cambridge University Press, 2015.

\bibitem{Ting12}
T.~Gao, J.~Duan, X.~Li, and R.~Song.
\newblock Mean exit time and escape probability for dynamical systems driven by
  {L}\'evy noise.
\newblock {\em SIAM J. Sci. Comput.}, 36(3):A887--A906, 2014.

\bibitem{Gardiner2004}
C.~W. Gardiner.
\newblock {\em Handbook of Stochastic Methods}.
\newblock Springer, 3rd Ed., 2004.

\bibitem{golub2012matrix}
G.~H. Golub and C.~F. Van~Loan.
\newblock {\em Matrix computations}.
\newblock JHU Press, 4th Ed., 2012.

\bibitem{GRUBB15}
G.~Grubb.
\newblock Fractional {L}aplacians on domains, a development of {H}ormander's
  theory of u-transmission pseudodifferential operators.
\newblock {\em Adv. Math.}, 268:478--528, 2015.

\bibitem{Mengli}
M.~Hao, J.~Duan, R.~Song, and W.~Xu.
\newblock Asymmetric non-{G}aussian effects in a tumor growth model with
  immunization.
\newblock {\em Appl. Math. Model.}, 38:4428--4444, 2014.

\bibitem{asymmetric_Hein}
C.~Hein, P.~Imkeller, and I.~Pavlyukevich.
\newblock Limit theorems for p-variations of solutions of {SDE}s driven by
  additive stable {L}{\'e}vy noise and model selection for paleo-climatic data.
\newblock {\em Interdiscip. Math. Sci.}, 8:137--150, 2009.

\bibitem{qiao16}
Q.~Huang, J.~Duan, and J.~Wu.
\newblock Maximum principles for nonlocal parabolic waldenfels operators.
\newblock {\em arXiv:1607.02836}, 2016.

\bibitem{Huang14}
Y.~Huang and A.~Oberman.
\newblock Numerical methods for the fractional {L}aplacian: A finite
  difference-quadrature approach.
\newblock {\em SIAM J. Numer. Anal.}, 52(6):3056--3084, 2014.

\bibitem{Nature10}
N.~Humphries~{\it et al.}
\newblock Environmental context explains {L}{\'e}vy and {B}rownian movement
  patterns of marine predators.
\newblock {\em Nat. Lett.}, pages 1066--1069, 2010.

\bibitem{first}
T.~Koren, A.~Chechkin, and J.~Klafter.
\newblock On the first passage time and leapover properties of {L}{\'e}vy
  mmotion.
\newblock {\em Physica A}, 379:10--22, 2007.

\bibitem{Kwasnicki17}
M.~Kwasnicki.
\newblock Ten equivalent definitions of the fractional {L}aplace operator.
\newblock {\em Fract. Calc. Appl. Anal.}, 20(1):7--51, 2017.

\bibitem{Liu2004}
F.~Liu, V.~Anh, and I.~Turner.
\newblock Numerical solution of the space fractional {F}okker-{P}lanck
  equation.
\newblock {\em J. Comput. Appl. Math.}, 166:209--219, 2004.

\bibitem{FPDE_Shen}
Z.~Mao and J.~Shen.
\newblock Efficient {S}pectral-{G}alerkin method for fractional partial
  differential equations with variable coefficients.
\newblock {\em J. Comput. Phys.}, 307:243--261, 2016.

\bibitem{Mao17}
Z.~Mao and J.~Shen.
\newblock Hermite spectral methods for fractional {PDEs} in unbounded domains.
\newblock {\em SIAM J. Sci. Comput.}, 39(5):A1928--A1950, 2017.

\bibitem{Riabiz17}
M.~Riabiz and S.~Godsill.
\newblock Approximate simulation of linear continuous time models driven by
  asymmetric stable {L}{\'e}vy processes.
\newblock {\em IEEE International Conference on Acoustics, Speech and Signal
  Processing (ICASSP), New Orleans, LA, 2017}, pages 4676--4680, 2017.

\bibitem{FPE_Risken}
H.~Risken.
\newblock {\em The Fokker-Planck Equation Methods of Solution and
  Applications}.
\newblock Springer-Verlag Berlin Heidelberg, 2nd edition, 1996.

\bibitem{ROSOTON14}
X.~Ros-Oton and J.~Serra.
\newblock The {D}irichlet problem for the fractional {L}aplacian: Regularity up
  to the boundary.
\newblock {\em J. Math. Pure. Appl.}, 101(3):275--302, 2014.

\bibitem{Taqqu}
G.~Samorodnitsky and M.~S. Taqqu.
\newblock {\em Stable Non-Gaussian Random Process}.
\newblock Chapman \& Hall/CRC, 1994.

\bibitem{Sato-99}
K.~I. Sato.
\newblock {\em L{\'e}vy Processes and Infinitely Divisible Distributions}.
\newblock Cambridge University Press, 1999.

\bibitem{Duan01}
D.~Schertzer, M.~Larcheveque, V.~V. Duan, J.and~Yanovsky, and S.~Lovejoy.
\newblock Fractional {F}okker-{P}lanck equation for nonlinear stochastic
  differential equations driven by non-{G}aussian {L}{\'e}vy stable noises.
\newblock {\em J. Math. Phys.}, 42(1):200--212, 2001.

\bibitem{Sidi}
A.~Sidi and M.~Israeli.
\newblock Quadrature methods for periodic singular and weakly singular fredholm
  integral equtaions.
\newblock {\em J. Sci. Comput.}, 3(2):201--231, 1998.

\bibitem{asy_envirment}
T.~Srokowski.
\newblock Asymmetric {L}{\'e}vy flights in nonhomogeneous environments.
\newblock {\em J. Stat. Mech.-Theory E.}, 2014(5):P05024, 2014.

\bibitem{TD15}
X.~Tian and Q.~Du.
\newblock Nonconforming discontinuous {G}alerkin methods for nonlocal
  variational problems.
\newblock {\em SIAM J. Numer. Anal.}, 53(2):762--781, 2015.

\bibitem{Wang17}
M.~Wang and J.~Duan.
\newblock Existence and regularity of a linear nonlocal {F}okker-{P}lanck
  equation with growing drift.
\newblock {\em J. Math. Anal. Appl.}, 449(1):228--243, 2017.

\bibitem{METXiao2015}
X.~Wang, J.~Duan, X.~Li, and Y.~Luan.
\newblock Numerical methods for the mean exit time and escape probability of
  two-dimensional stochastic dynamical systems with non-{G}aussian noises.
\newblock {\em Appl. Math. Comput.}, 258:282--295, 2015.

\bibitem{Xiao2016}
X.~Wang, J.~Duan, X.~Li, and R.~Song.
\newblock Numerical algorithms for mean exit time and escape probability of
  stochastic systems with asymmetric {L}{\'e}vy motion.
\newblock {\em arXiv:1702.00600}, 2017.

\bibitem{Wei15}
J.~Wei and R.~Tian.
\newblock Well-posedness for the fractional {F}okker-{P}lanck equations.
\newblock {\em J. Math. Phys.}, 56(3):1--12, 2015.

\bibitem{Xu2013}
Y.~Xu, J.~Feng, J.~Li, and H.~Zhang.
\newblock {L}{\'e}vy noise induced switch in the gene transcriptional
  regulatory system.
\newblock {\em Chaos}, 23(1):013110, 2013.

\bibitem{Zeng10}
L.~Zeng and B.~Xu.
\newblock Effects of asymmetric {L}{\'e}vy noise in parameter-induced aperiodic
  stochastic resonance.
\newblock {\em Physica A}, 389:5128--5136, 2010.

\end{thebibliography}
\end{document}